\documentclass[11pt]{article} 
 \usepackage{amssymb, amsmath, amsthm,amsrefs}
\usepackage{latexsym}
\usepackage{color}

\usepackage{graphics}
\usepackage{hyperref}
\usepackage{todonotes}
\usepackage{exscale}
\usepackage{pdfsync}
\usepackage{enumerate}
\usepackage{fullpage}
\usepackage{bm, bbm, mathrsfs}

  \newtheorem{theorem}{Theorem}[section]
\newtheorem{lemma}[theorem]{Lemma}
\newtheorem{proposition}[theorem]{Proposition}

\newtheorem{definition}{Definition}

\newtheorem{remark}{Remark}[section]

\DeclareMathOperator{\scl}{S}

\DeclareMathOperator{\diag}{diag}

\makeatletter
\newcommand\Aboxed[1]{
   \@Aboxed#1\ENDDNE}
\def\@Aboxed#1&#2\ENDDNE{%
   &
   \settowidth\@tempdima{$\displaystyle#1{}$}
   \setlength\@tempdima{\@tempdima+\fboxsep+\fboxrule}
   \kern-\@tempdima
   \boxed{#1#2}
}

\makeatother

\newcommand{\propref}[1]{Proposition~\ref{#1}}
\newcommand{\secref}[1]{Section~\ref{#1}}

\newcommand{\lemref}[1]{Lemma~\ref{#1}}

\newcommand{\cov}[2]{\nabla_{#1}{#2}} 

\newcommand{\bpm}{\begin{pmatrix}}
\newcommand{\met}[1]{\mathbf{#1}}
\newcommand{\epm}{\end{pmatrix}}

\newcommand{\meti}[1]{\omega^{#1} \otimes \omega^{#1}}

\newcommand{\lied}[2]{\mathcal{L}_{#2}{#1}}

\newcommand{\tp}[2]{\omega^{#1} \otimes \omega^{#2}}
\newcommand{\vf}[1]{\mathbf{#1}}

\newcommand{\tensor}[1]{\mathbf{#1}}
\newcommand{\ccf}{\tensor{H}}
\newcommand{\ccfop}{\widehat{\tensor{H}}}
\newcommand{\rg}{\tensor{RG}}
\newcommand{\rgop}{\widehat{\tensor{RG}}}
\newcommand{\ric}{\tensor{Rc}}
\newcommand{\ein}{\tensor{E}}
\newcommand{\bv}[1]{\vf{e}_{#1}}
\newcommand{\av}[1]{\tilde{\vf{e}}_{#1}}
\newcommand{\cv}[1]{\omega^{#1}}

\newcommand{\lp}{\left(}
\newcommand{\rp}{\right)}
\newcommand{\rcc}[2]{R_{{#1}{#2}}}

\begin{document}
\title{On algebraic solitons for geometric evolution equations on three-dimensional Lie groups }\author{T. H. Wears\\ Department of Mathematics and Computer Science\\ Longwood University\\ 201 High St\\ Farmville, VA 23909, USA}

\maketitle

\begin{abstract}
In this paper, we investigate the relationship between algebraic soliton metrics and soliton metrics for geometric evolution equations on Lie groups.
After discussing the general relationship between algebraic soliton metrics and soliton metrics, we investigate the cross curvature flow and the second order renormalization group flow  on simply connected three-dimensional unimodular Lie groups, providing a complete classification of left invariant algebraic solitons on such spaces.
\end{abstract}

\section{Introduction}

A major theme in modern geometry and topology is the use of geometric evolution equations to improve a certain geometric structure or geometric quantity on a smooth manifold $\mathcal{M}$.  
Perhaps the most celebrated example of this is the Ricci flow as introduced by Richard Hamilton in \cite{hamiltonintro}, where one views the Ricci flow equation
\begin{equation}
\label{ricciflow}
\frac{\partial \met{g}}{\partial t} = -2\ric\left[\met{g}\right] ; \hspace{.5in} \met{g}_{0} = \met{g}(0)
\end{equation}
as something akin to a heat equation for the evolution of the metric tensor $\met{g}$ on the underlying manifold structure.
Of particular interest in this situation is the evolution of geometric quantities associated with metric tensor $\met{g}$ and the singularities that form (in either finite or infinite time) for solutions of the corresponding system of partial differential equations.
Using a rescaling of the flow near, such singularities are typically modeled with self-similar solutions of the corresponding Ricci flow.
These self-similar solutions are what are known as \emph{Ricci solitons} and are characterized for the Ricci flow as solutions to \eqref{ricciflow} the form $\met{g}(t) = c(t)\phi^{*}_{t}\met{g}_{0}$, where $\phi_{t}$ is one-parameter family of diffeomorphsims of $\mathcal{M}$ and $c$ is a scalar valued function on $\mathcal{M}$.
Ricci solitons have been studied extensively since Hamilton's introduction of the Ricci flow.  
We refer the reader to  \cite{ChowBrown2},  \cite{ChowBrown1}, \cite{ChowBlue}, \cite{hamiltonsing}, and references therein for a further discussion of the role of Ricci solitons in the study of the Ricci flow.

In addition to playing a role in the study of the Ricci flow on closed manifolds, Ricci solitons have also played an important role in the study of preferred and/or distinguished metrics on Lie groups.  
Namely, in \cite{Lauret}, Lauret introduced Ricci solitons as a natural generalization to an Einstein metric for nilpotent Lie groups.
More to the point, in \cite{Milnor}, Milnor establishes that a nilpotent Lie group cannot carry a left invariant Einstein metric (i.e., a metric $\met{g}$ satisfying $\ric\left[\met{g}\right] = \lambda \met{g}$ for some scalar $\met{g}$).
Viewing an Einstein metric as a fixed point for the Ricci flow \eqref{ricciflow}, then a Ricci soliton can be suitably interpreted as a geometric fixed point for the flow as a Ricci soliton is a metric that  evolves only by scaling and diffeomorphism.
Furthermore, it is clear that that a Ricci soliton $\met{g}$ must satisfy the  Ricci Soliton equation 
	\begin{equation}
	\label{ricsolitonequation}
	 \beta \met{g} + \lied{\vf{X}}{\met{g}} = -2\ric\left[\met{g}\right],
	\end{equation}
where $\beta$ is a scalar,  $\vf{X}$ is a vector field on $\mathcal{M}$, and $\lied{\vf{X}}{\met{g}}$ is the Lie derivative of the metric $\met{g}$ in the direction of the vector field $\vf{X}$.
When $\met{g}$ is an Einstein metric, one can take the vector field $\vf{X}$ to be a Killing field of the metric $\met{g}$, whereas when $\met{g}$ is a non-trivial Ricci soliton, the vector field $\vf{X}$ appearing in \eqref{ricsolitonequation} will be unique up to a Killing field of the metric tensor $\met{g}$.

With this generalization in mind, Lauret \cite{Lauret} establishes that Ricci soliton metrics on nilpotent Lie groups can be found via algebraic methods alone. 
Namely, denoting the Ricci operator by  $\widehat{\ric\left[\met{g}\right]}$,  Lauret notes that any left invariant metric $\met{g}$ such that $\widehat{\ric\left[\met{g}\right]} - \kappa\mathrm{Id}$ is a derivation of the corresponding  Lie algebra (for some scalar $\kappa$) will give rise to a corresponding Ricci soliton.  
Lauret refers to such metrics as \emph{algebraic Ricci solitons}.
Lauret additionally shows that Ricci solitons on simply connected nilpotent Lie groups are unique up to a constant scalar multiple and an automorphism of the corresponding Lie algebra.  
See \cite{Lauret} for details. 

After Lauret introduced algebraic Ricci solitons, they have been studied extensively on Lie groups and homogeneous spaces by numerous authors.
It is interesting to note that to date, the only known examples of algebraic Ricci solitons on non-compact Lie groups occur on solvable Lie groups.
This situation is identical to that of the case of Einstein metrics.
See \cite{jablonski3}, \cite{jablonski2}, \cite{jablonski1}, and  \cite{lauret2} for further discussion.

In the current paper, our aim is twofold.  For starters, we aim to show that Lauret's ideas pertaining to algebraic solitons apply equally well to an arbitrary geometric evolution equation (subjection to the appropriate conditions) for a left-invariant Riemannian metric on a simply-connected Lie group.
This will build off of the work of Glickenstein in \cite{Glick} and Lauret in \cite{Lauret}.
We then aim to apply these ideas to the cross curvature flow and the second order renormalized group flow on simply connected three-dimensional unimodular Lie groups, where we classify those algebraic solitons that give rise to self-similar solutions of the respective flow. 
Note that in \cite{Glick}, Glickenstein investigates the cross curvature flow on three-dimensional unimodular Lie groups using Riemannian groupoids, and in \cite{Gimregeometric}, the authors discuss steady solitons to the second order renormalized group flow.
As it pertains to soliton metrics for the cross curvature flow, we will obtain similar results to Glickenstein, but we use entirely different methods.

The outline of the paper is as follows.  
In \secref{geomprelim} we introduce the cross curvature flow on three-dimensional manifolds (\secref{xcfintro}) and the second order renormalization group flow (\secref{rg2intro}).
In \secref{gensoliton}, we introduce the appropriate conditions that must be satisfied by a geometric evolution equation in order to tie algebraic solitons together with self-similar solutions to the corresponding flow.
We then follow this up with a brief discussion of some of the necessary algebraic conditions that a simply connected Lie group and its corresponding Lie algebra must satisfy in order to be able to support an algebraic soliton for a given geometric evolution equation for the metric tensor.

In \secref{algsolliegroup}, we begin by reviewing Milnor's construction of the so-called \emph{Milnor frames} on a simply connected three-dimensional unimodular Lie group equipped with a left invariant Riemannian metric.
We then review the geometric expressions and tensors associated with the cross curvature flow and the second order renormalized group flow as they appear in a Milnor frame.
This is followed in \secref{r3} - \secref{su2} by the classification of algebraic solitons that give rise to self-similar solutions of the cross curvature flow and second order renormalized group flow on simply connected three-dimensional unimodular Lie groups.

\section{Geometric preliminaries and notation}
	\label{geomprelim}
	\subsection{Geometric evolution equations on 3-manifolds}
	In this section we recall the definitions of the cross curvature flow $(\mathrm{XCF})$ and the second order renormalization group flow ($\mathrm{RG}$-$\mathrm{2}$) on a three-dimensional Riemannian manifold $(\mathcal{M}, \met{g})$.
	For general Riemannian manifolds, neither flow is well-posed and one is not able to guarantee short-time existence of solutions.
	However, if one restricts their attention to left invariant metrics on Lie groups  (or more generally to homogenous spaces), then short-time existence and uniqueness follows from the standard existence and uniqueness results for ordinary differential equations.
	On a Lie group, for example, the selection of a frame at the identity is a tantamount to selecting global coordinate functions on the fiber of the bundle of positive-definite symmetric (0,2)-tensors and the study of the geometric evolution equation in question reduces to the evolution of the coefficient functions of the metric tensor with respect to the selected frame.
	On a three-dimensional Lie group, this results in a system of six coupled ordinary differential equations, but we will see that much like the situation for the Ricci flow on three-dimensional homogeneous geometry, this can be to reduced system of three ordinary differential equations for both the $\mathrm{XCF}$ and $\mathrm{RG}$-$\mathrm{2}$ flow.

	For a more detailed discussion of the cross curvature flow, we refer the readers to \cite{xcfintro}, where the the flow was introduced by Chow and Hamilton  on three-dimensional manifolds with strictly positive or strictly negative sectional curvature, and to \cite{xcf1} and \cite{xcf2}, where the authors study the $\mathrm{XCF}$ on three-dimensional homogeneous geometries.
	Additionally, in \cite{Buckland}, Buckland establishes the existence of solutions to the cross curvature flow on 3-manifolds equipped with an initial Riemannian metric that has everywhere positive or everywhere negative sectional curvature.
	For examples of the $\mathrm{XCF}$ on a square torus bundle over $\mathbf{S}^{1}$ and on $\mathbf{S}^{2}$-bundles over $\mathbf{S}^{1}$, see \cite{MaChen}, and for a program trying to utilize the $\mathrm{XCF}$ for the purpose of trying to to show that the moduli space of negatively curved metrics on a closed hyperbolic $3$-manifold is path-connected, see \cite{Knopf}.
	
	For a more detailed discussion of the $\mathrm{RG}$-$\mathrm{2}$ flow we refer the reader to \cite{Gimregeometric}, where the authors focus on a  geometric introduction to the $\mathrm{RG}$-$\mathrm{2}$ flow, and \cite{Gimre3d}, where the authors study the $\mathrm{RG}$-$\mathrm{2}$ flow on three-dimensional homogeneous spaces in a spirit similar to the analysis of the Ricci flow on three-dimensional homogeneous spaces carried out by Isenberg and Jackson in \cite{Isenberg3}.
	
	\subsubsection{$\mathrm{XCF}$}
	\label{xcfintro}
		Before we introduce the positive and negative cross curvature flow $\lp \pm XCF\rp$ on a locally homogeneous three-dimensional manifold as defined in \cite{xcf1},  a few remarks are in order.
		 In \cite{xcfintro}, Chow and Hamilton considered the $\mathrm{XCF}$ flow on  three-dimensional Riemannian manifolds $\lp \mathcal{M}, \met{g}\rp$ where the sectional curvatures of the initial metric $\met{g}$ were strictly positive or strictly negative, and the sign of the of the $\mathrm{XCF}$ was chosen depending on the sign of the sectional curvatures ($\pm$ $\mathrm{XCF}$ when the sign of the sectional curvatures of $\met{g}$ are $\mp$).
		 One does not expect to obtain general existence results for the $\mathrm{XCF}$, but as noted above, by restricting one's attention to Lie groups or (locally) homogeneous spaces, short-time existence of solutions is easily obtained from standard ODE results (regardless of the signs of the sectional curvatures).
		Following  \cite{xcf1} and \cite{xcfintro} we will now define the cross curvature tensor and the positive and negative $\mathrm{XCF}$ on locally homogeneous manifolds.\\
	
		The following construction/definition of the cross curvature tensor is taken from \cite{xcfintro}.
		Let $(\mathcal{M}, \met{g})$ be a three-dimensional Riemannian manifold, with corresponding Ricci tensor $\ric\left[\met{g}\right] = \ric$ and scalar curvature $\scl\left[\met{g}\right] = \scl$.
		The component functions of the Ricci tensor with respect to a given local frame will be denoted by $\rcc{i}{j}$.
		Denote the corresponding Einstein tensor by $\ein\left[\met g\right] = \ein = \ric - \frac{\scl}{2}\met{g}$ and define a $(2,0)$-tensor $\tensor{P}\left[\met{g}\right] = \tensor{P}$ by raising the indices on the Einstein tensor $\ein$.  
		The component functions of $\tensor{P}$ with respect to a local frame are thusly
			\begin{equation*}
				P^{ij} = g^{ik}g^{jl}\rcc{kl}   -  \frac{\scl}{2}g^{ij}.
			\end{equation*}
		Provided that $\lp P^{ij}\rp$ is invertible, we denote $\lp P^{ij}\rp^{-1}$ by $\lp V_{ij}\rp$ and then define the cross curvature tensor $\ccf$ by the component functions
			\begin{equation}
				H_{ij} = \lp \frac{\det P^{ij}}{\det g^{ij}}\rp V_{ij}.
			\end{equation}
		As observed in  \cite{xcf1} and \cite{MaChen}, the cross curvature tensor $\ccf$ takes a particularly simple form in a local orthonormal frame $\bv{1}$, $\bv{2}$, $\bv{3}$ where the Ricci tensor $\ric$ is diagonalized.
		When $\lp l, m, n\rp$ is any permutation of $\lp 1, 2, 3\rp$,  let  $K_{l} = K\lp \bv{m} \wedge \bv{n}\rp$  denote the principal sectional curvature of the plane that is orthogonal to $\bv{l}$.  
		We thus have $\rcc{l}{l} = K_{m} + K_{n}$,
		and it follows that the cross curvature tensor $\ccf$ is diagonalized with respect to the indicated frame and the component functions of $\ccf$ are
			\begin{equation}
			\label{xcfcompgen}
				H_{ll} =  K_{m}K_{n}, \hspace{.5in} \textrm{$\lp l, m, n\rp$ a permutation of $\lp 1, 2, 3\rp$.}
			\end{equation}
		Note that \eqref{xcfcompgen} defines the cross curvature tensor when $\tensor{P}\left[\met{g}\right]$ fails to be invertible.
		Additionally, note that $\ccf$ obeys the following important homogeneity property for a scaling of the metric tensor $\met{g}$:  
			\begin{equation}
			\label{scaling}
				\ccf\left[ c \met{g}\right] = \frac{1}{c} \ccf\left[\met{g}\right],  \quad c \in \mathbb{R}_{>0}.
			\end{equation}
		
		Following \cite{xcf1}, we now define the  $\mathrm{XCF}$ on locally homogeneous three-dimensional manifolds.

			\begin{definition}[Cross Curvature Flow]
			\label{XCF}
				Let $(\mathcal{M}, \met{g}_{0})$ be a locally homogeneous three-dimensional manifold.
				The \emph{positive cross curvature flow} (+XCF) is defined by
					\begin{equation*}
						\frac{\partial \met{g}}{\partial t} = 2 \ccf\left[ \met{g}\right] ; \hspace{.5in} \met{g}(0) = \met{g}_{0}
					\end{equation*}
				and the \emph{negative cross curvature flow} (-XCF) is defined by
					\begin{equation*}
						\frac{\partial \met{g}}{\partial t} = -2 \ccf\left[ \met{g}\right] ; \hspace{.5in} \met{g}(0) = \met{g}_{0}
					\end{equation*}
			\end{definition}
		Due to original use of the cross curvature flow by Chow and Hamilton, there is not a clear choice as to what should be the forward direction for the $\mathrm{XCF}$ when the signs of the sectional curvature vary and both directions of the flow on three-dimensional homogeneous spaces are investigated in  \cite{xcfsl2}, \cite{xcf1}, and \cite{xcf2} in a similar spirit to the analysis carried out by Isenberg and Jackson for the Ricci flow on homogeneous three-dimensional geometries in their seminal paper \cite{Isenberg3}.

		\subsubsection{The $\mathrm{RG}$-2 Flow}
		\label{rg2intro}
		 
		 We will now briefly introduce the $\mathrm{RG}$-$\mathrm{2}$ flow.		 
		 The $\mathrm{RG}$-$\mathrm{2}$flow is a second order approximation of the Renormalization Group flow that corresponds to a perturbative analyses of nonlinear sigma model quantum field theories from a world sheet into $\lp\mathcal{M}, \met{g}\rp$, and unlike the $\mathrm{XCF}$, which has currently only been defined and investigated on three-dimensional Riemannian manifolds, the $\mathrm{RG}$-$\mathrm{2}$ flow has been studied in arbitrary dimensions.
		 For an introduction to the physics of the Renormalization Group flow, we refer the reader to \cite{Friedan1} and \cite{Gawd}.  
		 We will mostly follow the geometrical introduction to the Renormalization Group flow as found in the work of Gimre, Guenther, and Isenberg in \cite{Gimregeometric} and \cite{Gimre3d}.

		 The evolution of the metric tensor $\met{g}$ under the nonlinear sigma quantum field theories takes the form
		 	\begin{equation}
			\label{renormalization}
			\frac{\partial \met{g}}{\partial t} = - \alpha \ric\left[\met{g}\right] - \frac{\alpha^{2}}{2} \tensor{Rm}^{2}\left[\met{g}\right] + \mathcal{O}\lp \alpha ^3 \rp,
			\end{equation}
		where $\alpha$ denotes a positive coupling constant and the tensor $\tensor{Rm}^{2} \left[\met{g}\right] = \check{R}_{ij}\tp{i}{j}$ involves quadratic terms stemming from the full Riemannian curvature tensor of the metric tensor $\met{g}$.
		The component functions of $\tensor{Rm}^{2} \left[\met{g}\right] = \check{R}_{ij}\tp{i}{j}$ are
			\[
			\check{R}_{ij} = Rm_{iklm}Rm_{jpqr}g^{kp}g^{lq}g^{mr},
			\]
		where $\tensor{Rm} = Rm_{ijkl}\tp{i}{j}\otimes\tp{k}{l}$ denotes the Riemannian curvature tensor of $\met{g}$ and $g^{ij}$ denote the component functions of $\met{g}^{-1}$.
		After an appropriate rescaling of the parameter $t$, one finds that the second order approximation of Renormalization group flow takes the form
			\begin{equation}
			\label{renormalization2}
			\frac{\partial \met{g}}{\partial t} = - 2 \ric\left[\met{g}\right] - \frac{\alpha}{2} \tensor{Rm}^{2} \left[\met{g} \right].
			\end{equation}
		 We will denote the tensor  $-2 \ric\left[\met{g}\right] - \frac{\alpha}{2} \tensor{Rm}^{2} \left[\met{g} \right]$ by $\rg\left[\met{g}\right]$ and we define the second order Renormalization Group flow (the $\mathrm{RG}$-$\mathrm{2}$ flow) for a Riemannian metric on a smooth manifold $\mathcal{M}$ as follows.
		 \begin{definition}[$\mathrm{RG}$-$\mathrm{2}$ Flow]
		 Let $\lp \mathcal{M}, \met{g}_{0}\rp$ be a smooth Riemannian manifold.  
		 The \emph{second order Renormalization Group ($\mathrm{RG}$-$\mathrm{2}$)  flow} is the geometric evolution equation for the Riemannian metric $\met{g}_{0}$ defined by 
		 	\begin{equation}
			\label{rg2floeqn}
				\frac{\partial \met{g}}{\partial t} = \rg\left[\met{g}\right] =  - 2 \ric\left[\met{g}\right] - \frac{\alpha}{2} \tensor{Rm}^{2} \left[\met{g} \right] ; \hspace{.5in} \met{g}(0) = \met{g}_{0}.
			\end{equation}
		 \end{definition}
		The  two terms comprising the tensor $\rg\left[\met{g}\right]$ behave in the following manner under a positive scaling of the metric tensor $\met{g}$: $\ric\left[c \met{g}\right] =  \ric \left[\met{g} \right]$  and $\tensor{Rm}^{2} \left[c \met{g} \right] = \frac{1}{c} \tensor{Rm}^{2}$.
 		Additionally, note in particular that the first order approximation of the Renormalization Group flow is the Ricci flow and that \eqref{rg2floeqn} can be studied mathematically as a non-linear deformation of the Ricci flow, which is precisely the point of view that is taken in \cite{Gimregeometric} and \cite{Gimre3d}.
		
		In \cite{Gimre3d}, the authors study the $\mathrm{RG}$-$\mathrm{2}$ flow on three-dimensional homogenous geometries.
		In particular the authors study how the asymptotic behavior of the $\mathrm{RG}$-$2$ flow depends on the parameter $\alpha$ and how it compares with the asymptotic flow of the Ricci flow, where they use the results of \cite{Isenberg3} as their benchmark and guide.
		The short-time existence and uniqueness problem for the $\mathrm{RG}$-$\mathrm{2}$ flow is unsettled for a general Riemannian manifold.  However, in \cite{Gimreexists}, the authors establish the existence and short-time uniqueness for the second order Renormalization Group flow initial value problem on closed Riemannian manifolds $\lp \mathcal{M}, \met{g}_{0}\rp$ in general dimensions in the case where the sectional curvatures  of the  initial metric $\met{g}_{0}$ satisfy $1 + \alpha K_{P} > 0$ at all points $p \in \mathcal{M}$ and two-planes $P \subset T_{p} \mathcal{M}$.
		As noted above, we remind the reader that the existence and uniqueness of solutions to the $\mathrm{RG}$-$\mathrm{2}$ flow on Lie groups (or more generally homogeneous geometries) is guaranteed as the flow reduces to a system of ordinary differential equations for the inner product on a chosen product.

		  \subsubsection{Solitons}
		  \label{gensoliton}
		 We now make several general observations concerning the $\mathrm{XCF}$ and the $\mathrm{RG}$-$\mathrm{2}$ flow and the relationship between self-similar solutions of the flows and so-called solitons.
		 For simplicity, we will assume that $\lp\mathcal{M}, \met{g}\rp$ is a three-dimensional Riemannian manifold.
		 Let $\tensor{T}: \mathrm{Sym}^{2}\left( T^{*}\mathcal{M}\right) \to \mathrm{Sym}^{2}\lp T^{*} \mathcal{M}\rp$ be a fiber preserving map of the bundle of symmetric (0,2)-tensors on $\mathcal{M}$ and consider a geometric evolution equation on $\mathcal{M}$ of the form
		 	\begin{equation}
			\label{generalflow}
				\frac{\partial \met{g}}{\partial t } = \tensor{T}\left[ \met{g}\right]; \hspace{.5in} \met{g}(0) = \met{g}_{0}.
			\end{equation}
		Observe, for example,  that the Ricci flow, the $\mathrm{RG}$-$\mathrm{2}$ flow and the $\mathrm{XCF}$ are all of the indicated form.
		Following \cite{Glick}, we now take note of the following important properties that can be satisfied by $\tensor{T}$.
			\begin{definition}[Natural and Homogeneous]
			\label{natural}
			We say that $\tensor{T}: \mathrm{Sym}^{2}\left( T^{*}\mathcal{M}\right) \to \mathrm{Sym}^{2}\lp T^{*} \mathcal{M}\rp$  is 
				\begin{enumerate} 
					\item \emph{natural} if for all diffeomorphisms $\varphi : \mathcal{M} \to \mathcal{M}$ and all Riemannian metrics $\met{g}$ on $\mathcal{M}$, $\tensor{T}$ satisfies $\varphi^{*} \tensor{T}\left[ \met{g}\right] = \tensor{T}\left[\varphi^{*} \met{g}\right]$, and
					
					\item \emph{homogeneous of degree $q$} if for any scaling $c \met{g}$ of a Riemannian metric $\met{g}$, $\tensor{T}$ satisfies $\tensor{T}\left[c \met{g}\right] = c^{q} \tensor{T}\left[\met{g}\right]$.
				\end{enumerate}
			\end{definition}
			
			\begin{remark}
			Note that $\ric$, $\ccf$ and $\rg$ are all natural, while $\ric$ is homogeneous of degree zero and $\ccf$ is homogeneous of degree $p = -1$. $\rg$ fails to be homogeneous as it is comprised of two homogeneous terms of different degrees.
			Further, note that the definition of natural amounts to saying that the symmetry group of the evolution equation \eqref{generalflow} contains the full diffeomorphism group of $\mathcal{M}$.
			In the case where  $\lp \mathcal{M}, \met{g}\rp$ is a homogeneous space  it then follows that $\tensor{T}\left[\met{g}\right]$ is invariant under the group acting on $\mathcal{M}$.
			In particular, if $\tensor{T}$ is natural  and $\met{g}$ is a left invariant metric on a Lie group, then $\tensor{T}\left[\met{g}\right]$ will be left invariant as well.
			\end{remark}			
			
			Given a geometric evolution equation on $\mathcal{M}$ of the form \eqref{generalflow},  a solution $\met{g}(t)$ that evolves only by scaling and diffeomorphism is said to be a \emph{self-similar solution}. 
			Thus a self-similar solution is of the form $\met{g}(t) = c(t)\varphi_{t}^{*}\met{g}_{0}$, where $\varphi_{t}$ is a one-parameter family of diffeomorphisms of $\mathcal{M}$ with $\varphi_{0} = \mathrm{Id}$ and $c(t)$ is a real-valued function satisfying $c(0) =1$.
			Such a solution should be regarded as a geometric fixed point for the flow \eqref{generalflow}.
			
			Related to self-similar solutions of the flow will be the so called \emph{$\tensor{T}$-soliton structures}.
			The quadruple $\left(\mathcal{M}, \met{g}_{0}, \vf{X}, \beta \right)$, where $\vf{X}$ is a vector field on $\mathcal{M}$ and $\beta \in \mathbb{R}$,  is said to be a \emph{$\tensor{T}$-soliton structure} for the geometric evolution equation \eqref{generalflow} if
				\begin{equation}
				\label{solitoneq}
					\beta \met{g}_{0} + \lied{\met{g}_{0}}{\vf{X}} = \tensor{T}\left[ \met{g}_{0}\right],
				\end{equation}
			where $\lied{\met{g}_{0}}{\vf{X}}$ denotes the Lie derivative of the metric $\met{g}_{0}$ in the direction of the vector field $\vf{X}$.
			 $\left(\mathcal{M}, \met{g}_{0}, \vf{X}, \beta \right)$ is said to be \emph{expanding} if $\beta$ is positive, \emph{steady} if $\beta $ is zero, and \emph{shrinking} if $\beta $ is negative.
			 We will  refer to the metric $\met{g}_{0}$ from the quadruple  $\left(\mathcal{M}, \met{g}_{0}, \vf{X}, \beta \right)$ as a $\tensor{T}$-soliton, emphasizing $\vf{X}$ and $\beta$ only when needed. 
			 As per usual, note that the vector field $\vf{X}$ in the soliton structure  $\left(\mathcal{M}, \met{g}_{0}, \vf{X}, \beta \right)$ is unique up to the addition of a Killing field of the metric tensor $\met{g}_{0}$.
			 
			 It is tempting to use the $\tensor{T}$-soliton equation to define self-similar solutions to \eqref{generalflow}, but as the following proposition shows, this requires that $\tensor{T}$ be natural and homogeneous.
			The key link between self-similar solutions of the evolution equation \eqref{generalflow} and soliton metrics was first established in \cite{Glick} and is a straight forward generalization of the relationship between Ricci solitons and self-similar solutions of the Ricci flow.
		
			\begin{proposition}
			\label{solitonselfsimilar}
				Let $\tensor{T} : \mathrm{Sym}^{2}\left(T^{*}\mathcal{M}\right) \to \mathrm{Sym}^{2}\left(T^{*}\mathcal{M}\right)$ be fiber preserving and assume that $\tensor{T}$ is natural.
					\begin{enumerate}
						\item  If $\met{g}(t) = c(t)\varphi^{*}_{t}\met{g}_{0}$ is a self-similar solution to the geometric evolution equation \eqref{generalflow}, then $\met{g}_{0}$ is a $\tensor{T}$-soliton.
						\item If $\met{g}_{0}$ is a steady $\tensor{T}$-soliton, then there is a self-similar solution to \eqref{generalflow} of the form $\met{g}(t) = \varphi_{t}^{*}\met{g}_{0}$ (i.e., $\met{g}_{0}$ evolves by diffeomorphisms).
						\item If $\met{g}_{0}$ is a $\tensor{T}$-soliton for \eqref{generalflow} and, in addition, $\tensor{T}$ is homogeneous of degree $q$, then there is a self-similar solution to \eqref{generalflow} of the form $\met{g}(t) = c(t)\varphi_{t}^{*}\met{g}_{0}$.
					\end{enumerate}
			\end{proposition}
			 
			 \begin{proof}
			 	\begin{description}
				\item $\mathit{1.}$
					Assume that $\met{g}(t) = c(t) \varphi^{*}_{t} \met{g}_{0}$ is a self-similar solution to \eqref{generalflow} and let $\vf{X}$ be the vector field on $\mathcal{M}$ defined by 
					$\vf{X}(p) = \frac{d}{dt}\left(\varphi_{t}(p)\right)$, $p \in \mathcal{M}$.
					It follows from properties of Lie derivatives that
						\begin{align*}
							\frac{\partial\met{g}(t)}{\partial t} &= c^{\prime}(t) \varphi_{t}^{*}\met{g}_{0} + c(t)\varphi^{*}_{t}\left(\lied{\met{g}_{0}}{\vf{X}}\right)&\\
												&= \varphi^{*}_{t}\lp c^\prime(t) \met{g}_{0} + c(t) \lied{\met{g}_{0}}{\vf{X}}\rp  &  \textrm{(Properties of pull-backs)} \\
												&= \tensor{T}\left[c(t) \varphi_{t}^{*}\met{g}_{0}\right]  & \textrm{( Since $\met{g}(t) = c(t)\varphi_{t}^{*}\met{g}_{0}$ is a solution to \eqref{generalflow})}\\
												&= \varphi_{t}^{*}\tensor{T} \left[ c(t) \met{g}_{0}\right]  &\textrm{($\tensor{T}$ is natural)}.\\
						\end{align*}
					Evaluating the above at time $t = 0$, we find 
					$c^{\prime}(0) \met{g}_{0} + \lied{\met{g}_{0}}{\vf{X}} = \tensor{T} \left[ \met{g}_{0}\right]$
					 and we conclude that $\met{g}_{0}$ is a $\tensor{T}$-soliton for the geometric flow \eqref{generalflow}.
				
				\item $\mathit{2.}$ Let $\vf{X}$ be a vector field such that $\lied{\met{g}_{0}}{\vf{X}} = \tensor{T}\left[\met{g}_{0}\right]$ and let $\varphi_{t}$ be the one-parameter family of diffeomorphisms generated by $\vf{X}$.
				It follows  that 
				\[
				\frac{\partial \varphi_{t}^{*}\met{g}_{0}}{\partial t} 
				= \varphi^{*}_{t}\left(\lied{\met{g}_{0}}{\vf{X}}\right) 
				= \varphi^{*}_{t} \tensor{T}\left[\met{g}_{0}\right] 
				= \tensor{T}\left[ \varphi^{*}_{t} \met{g}_{0}\right],
				\] and we conclude that $\met{g}(t) = \varphi^{*}_{t}\met{g}_{0}$ is a solution of \eqref{generalflow}.
				
				\item $\mathit{3.}$ Now, we assume that $\lp \mathcal{M}, \met{g}_{0}, \vf{X}, \beta \rp$ is a $\tensor{T}$-soliton structure for \eqref{generalflow} and additionally that $\tensor{T}$ is homogeneous of degree $q$.
				Define $c(t)$ to be the solution the differential equation $\frac{dc}{dt} = \beta c^{q};  c(0) = 1$ and define the time-dependent vector field $\vf{Y}_{t}$ by $\vf{Y}_{t} = c(t)^{q - 1} \vf{X}$.
				Denote the corresponding flow of $\vf{Y}_{t}$ by $\varphi_{t}$ and observe that if $\met{g}(t) = c(t) \varphi_{t}^{*} \met{g}_{0}$, then it follows from properties of Lie derivatives that
					\begin{align*}
							\frac{\partial \met{g}(t)}{\partial t} &= c^{\prime}(t) \varphi^{*}_{t}\met{g}_{0} + c(t)\varphi^{*}_{t}\left(\lied{\met{g}_{0}}{\vf{Y}}\right) & \medskip\\
												&=c^{\prime}(t) \varphi_{t}^{*}\met{g}_{0} + c(t)\varphi^{*}_{t}\left(\lied{\met{g}_{0}}{c(t)^{q - 1} \vf{X}}\right) & \lp \textrm{Defn. of } \vf{Y}\rp \medskip \\
												&= \beta c(t)^{q} \varphi_{t}^{*} \met{g}_{0} + c(t)^{q} \varphi^{*}_{t}\left(\lied{\met{g}_{0}}{\vf{X}}\right)  & \textrm{(Assumptions on $c(t)$)}\medskip\\
												\medskip																
												&=c(t)^{q}\varphi^{*}_{t}\lp \beta \met{g}_{0} +\lied{\met{g}_{0}}{\vf{X}}\rp & \textrm{( Properties of pull-backs )} \\
												\medskip
												&= c(t)^{q}\varphi^{*}_{t}\lp \tensor{T}\left[\met{g}_{0}\right]\rp & \textrm{($\lp \mathcal{M}, \met{g}_{0}, \vf{X}, \beta\rp $ is a $\tensor{T}$-soliton)} \\
												\medskip
												&= \tensor{T}\left[ c(t)\varphi^{*}_{t} \met{g}_{0}\right], &\\
					\end{align*}
				and we conclude that $\vf{T}$-solitons give rise to self-similar solutions of the flow \eqref{generalflow} whenever $\tensor{T}$ is natural and homogeneous.
				\end{description}
			 \end{proof}

			 Following the work of Lauret on the Ricci flow on nilpotent Lie groups in \cite{Lauret}, we will now define algebraic $\tensor{T}$-solitons for the geometric evolution equation \eqref{generalflow} and establish the relationship between algebraic $\tensor{T}$-solitons and $\tensor{T}$-solitons.
		Similar considerations apply in all dimensions, but for ease of exposition we restrict ourselves to dimension three and assume that
		$\lp \mathcal{M}, \met{g}\rp = \lp\mathcal{H}, \met{g}\rp$ is a simply connected three-dimensional Lie group with Lie algebra $\mathfrak{h}$ and left invariant Riemannian metric $\met{g}$.
		
		For a symmetric $(0, 2)$-tensor $\tensor{A}$ on $\lp \mathcal{M}, \met{g}\rp$, we refer to the $(1, 1)$-tensor $\widehat{\tensor{A}}$ obtained by using $\met{g}$ to raise an index as the \emph{$\tensor{A}$-operator}. 
		Note that $\widehat{\tensor{A}}$ is defined implicitly by requiring that 
		$\met{g}\left(\widehat{\tensor{A}}\left( \vf{X}\right), \vf{Y}\right) = \tensor{A}\left(\vf{X}, \vf{Y}\right)$ 
		for all vector fields $\vf{X}, \vf{Y}$ on $\mathcal{M}$, and that with respect to a given frame, the components $\widehat{A}^{i}_{j}$ of $\widehat{\tensor{A}}$ are related to the components 
		$A_{ij}$ of $\tensor{A}$ by $\widehat{A}^{i}_{j} = g^{ik}A_{kj}$.
		While $\widehat{\tensor{A}}$ depends on both $\tensor{A}$ and the $\met{g}$, the metric $\met{g}$ being used to define $\widehat{\tensor{A}}$ will always be clear from the context.
		
			\begin{definition}[Algebraic $\tensor{T}$-soliton]
			\label{algsol}
				Let $\mathcal{H}$ be  a simply connected Lie group with Lie algebra $\mathfrak{h}$ and  left invariant metric  $\met{g}$.
				The triple $\lp \mathcal{H},  \met{g}, \kappa \rp $, where $\kappa \in \mathbb{R}$,  is said to be an \emph{algebraic $\tensor{T}$-soliton} for the geometric evolution equation \eqref{generalflow} if  the operator $\tensor{D} : \mathfrak{h} \to \mathfrak{h}$ defined by
					 \begin{equation}
					 \label{algsoleq}
					 	 \tensor{D} = \widehat{\tensor{T}\left[\met{g}\right]}  - \kappa \mathrm{Id}
					 \end{equation}
				 is a derivation of $\mathfrak{h}$.
				 When there is no potential for confusion, we we will simply refer to the left invariant metric $\met{g}$ as an algebraic $\tensor{T}$-soliton.
			\end{definition}
		The following  proposition  is a straightforward adaptation from  \cite{Lauret} and  establishes the relationship between algebraic $\tensor{T}$-solitons and $\tensor{T}$-solitons.
			\begin{proposition}
			\label{algimpliessol}
			If $\mathcal{H}$ is a simply connected Lie group and the left invariant metric $\met{g}$ is an algebraic $\tensor{T}$-soliton for the geometric evolution equation \eqref{generalflow}, then $\met{g}$ is a $\tensor{T}$-soliton.			
			\end{proposition}
			\begin{proof}
				Assume that $\met{g}$ is an algebraic $\tensor{T}$-soliton with soliton constant of $\kappa$ and that $\bv{1}, \bv{2}, \bv{3}$ is an orthonormal basis for $\met{g}$.
				Set $\mathbf{D} = \widehat{\tensor{T}\left[\met{g}\right]} - \kappa \mathrm{Id}$ and
				define $\varphi_{t} : \mathcal{H} \to \mathcal{H}$ by setting $d\varphi_{t} = \mathrm{exp}\left(\frac{t\mathbf{D}}{2}\right)$.
				Now define a vector field $\mathbf{X}$ on $\mathcal{H}$ by setting $\vf{X}(p)= \frac{d \varphi_{t}(p)}{dt} \Big \vert_{t = 0}$, $p \in \mathcal{H}$.
				By properties of Lie derivatives, tt follows that 
					\begin{equation*}
					\lied{\met{g}}{\vf{X}} \left(\bv{i}, \bv{j}\right) =  \frac{d}{dt} \varphi_{t}^{*} \met{g}\left(\bv{i}, \bv{j}\right) \\
					= \frac{1}{2}\left(\met{g}\left(\mathbf{D}\left(\bv{i}\right), \bv{j}\right) + \met{g}\left(\bv{i}, \mathbf{D}\left(\bv{j}\right)\right)\right).
					\end{equation*}
				From the defining characteristics of $\mathbf{D}$, we find that 
					\begin{equation*}
					\frac{1}{2}\left(\met{g}\left(\mathbf{D}\left(\bv{i}\right), \bv{j}\right) + \met{g}\left(\bv{i}, \mathbf{D}\left(\bv{j}\right)\right)\right) 
					= \frac{1}{2}\left(
					\met{g}\left(\widehat{\tensor{T}\left[\met{g}\right]}\left(\bv{i}\right), \bv{j}\rp + \met{g}\left(\widehat{\tensor{T}\left[\met{g}\right]}\left(\bv{j}\right), \bv{i}\rp\rp - \kappa \met{g}\lp\bv{i}, \bv{j}\rp\\
					\end{equation*}
					 By definition of $\widehat{\tensor{T}\left[\met{g}\right]}$ and the symmetry of the metric tensor $\met{g}$, we also have that 
					\begin{align*}
						\tensor{T}\left[\met{g}\right]\left(\bv{i}, \bv{j}\rp - \kappa \met{g}\lp \bv{i}, \bv{j}\rp 
						&= \frac{1}{2}\lp \tensor{T}\left[\met{g}\right]\left(\bv{i}, \bv{j}\rp + \tensor{T}\left[\met{g}\right]\left(\bv{j}, \bv{i}\rp\rp  - \kappa \met{g}\lp \bv{i}, \bv{j}\rp\\ 
						& = \frac{1}{2}\left(
					\met{g}\left(\widehat{\tensor{T}\left[\met{g}\right]}\left(\bv{i}\right), \bv{j}\rp + \met{g}\left(\widehat{\tensor{T}\left[\met{g}\right]}\left(\bv{j}\right), \bv{i}\rp\rp - \kappa \met{g}\lp\bv{i}, \bv{j}\rp\\
					\end{align*}
			Thus we have that $ \lied{\met{g}}{\vf{X}} \left(\bv{i}, \bv{j}\right) = \tensor{T}\left[\met{g}\right]\left(\bv{i}, \bv{j}\rp - \kappa \met{g}\lp \bv{i}, \bv{j}\rp$ holds for all basis vectors $\bv{i}, \bv{j}$, and  we conclude that  $\lied{\met{g}}{\vf{X}} =  \tensor{T}\left[\met{g}\right] - \kappa \met{g}$ and that $\met{g}$ is a $\tensor{T}$-soliton for \eqref{generalflow}.
			\end{proof}
			\begin{remark}
			It  follows from \propref{solitonselfsimilar} that if $\tensor{T}$ is natural and homogeneous, then every algebraic $\tensor{T}$-soliton gives rise to a self-similar solution of the geometric evolution equation $\frac{\partial \met{g}}{\partial t} = \tensor{T}\left[\met{g}\right]$ with soliton constant $\kappa$.
			In light of this observation, we say that the algebraic $\tensor{T}$-soliton $\lp \mathcal{H}, \met{g}, \kappa\rp$ is expanding, steady, or shrinking, depending on whether or not $\kappa$ is positive, zero, or negative, respectively.
			Note that if $\tensor{T}$ is natural but not homogeneous, then a steady algebraic $\tensor{T}$-soliton gives rise to self-similar solution that evolves by diffeomorphism only.
			\end{remark}
			
			 \begin{remark}
		Note that  $\met{g}$ is a shrinking (res. expanding) soliton for the $\mathrm{+XCF}$ if and only if $\met{g}$ is an expanding (res. shrinking) soliton for the $\mathrm{-XCF}$.
		Specifically, $\lp \mathcal{M}, \met{g}, \vf{X}, \beta \rp$ is a $\mathrm{+XCF}$-soliton structure on $\mathcal{M}$ if and only if $\lp \mathcal{M}, \met{g}, -\vf{X}, -\beta \rp$ is a $\mathrm{-XCF}$ soliton structure on $\mathcal{M}$.
		This follows directly from the fact that  $\beta \met{g} + \lied{\met{g}}{\vf{X}} = 2 \ccf\left[\met{g}\right]$ if and only if $-\beta \met{g} + \lied{\met{g}}{-\vf{X}}= - 2 \ccf\left[\met{g}\right]$.
		Further note that when looking for soliton structures for the evolution equation $\frac{\partial \met{g}}{\partial t} = \tensor{T}\left[\met{g}\right]$, one can replace $\tensor{T}$ with any positive scalar multiple of $\tensor{T}$ without changing the the qualitative nature of the soliton structure.		
		In our classification of $\mathrm{XCF}$-algebraic soliton structures on three-dimensional unimodular Lie groups, we will use $\tensor{T} = \pm \ccf$ as opposed to $\tensor{T} = \pm 2 \ccf$. 
		\end{remark}
			
		Whereas the search for $\tensor{T}$-solitons often involves a complicated system of partial differential equations, algebraic $\tensor{T}$-solitons can be found via algebraic methods alone. 
		However, there are algebraic conditions that must be satisfied in order for a Lie group  $\mathcal{H}$ to support an algebraic $\tensor{T}$-soliton.
	The following proposition shows that the corresponding Lie algebra must have a derivation that can be diagonalized with a basis of eigenvectors.

	\begin{proposition}
	Let $\mathcal{H}$ be a simply connected Lie group with Lie algebra 
	$\mathfrak{h}$.
	If $\mathcal{H}$ admits an algebraic $\tensor{T}$-soliton structure for the geometric evolution equation \eqref{generalflow} and $\tensor{T}$ is natural, then $\mathfrak{h}$ admits a derivation that is diagonalizable.
	\end{proposition}
	
	\begin{proof}
	Suppose that $\left(\mathcal{H}, \met{g}, \kappa\right)$ is an algebraic $\tensor{T}$-soliton structure on $\mathcal{H}$ and note that the assumption of $\tensor{T}$ being natural implies that $\tensor{T}\left[\met{g}\right]$ is left invariant.
	Identifying both $\met{g}$ and $\tensor{T}\left[\met{g}\right]$ with their values on $\mathfrak{h} \simeq T_{e}\mathcal{H}$, then 
	since $\met{g}$ is positive-definite, there exists a basis $ \mathcal{B} = \left\{\bv{i} \right\}$ for $\mathfrak{h}$  that diagonalizes both $\met{g}$ and $\tensor{T}\left[\met{g}\right]$. 
	It  follows that $\widehat{\tensor{T}\left[\met{g}\right]}$ is also diagonalized with respect the indicated basis.
	The assumption that $\met{g}$ is an algebraic $\tensor{T}$-soliton implies that $\tensor{D} = \widehat{\tensor{T}\left[\met{g}\right]} - \kappa \mathrm{Id}$ is a diagonal derivation of the Lie algebra $\mathfrak{h}$ and that the basis $\mathcal{B} = \left\{ \bv{i} \right\}$ serves as a basis of eigenvectors for $\tensor{D}$.
	\end{proof}
	The following simple lemma concerning diagonal derivations  will be used extensively in the classification of $\mathrm{XCF}$ and $\mathrm{RG}$-$\mathrm{2}$ algebraic solitons  that give rise to self-similar solutions of their respective flows on three-dimensional unimodular Lie groups.
	We will state the lemma in arbitrary dimensions.
		\begin{lemma}
		\label{evlem}
		Let $\mathfrak{h}$ be a Lie algebra and suppose that $\mathbf{D}: \mathfrak{h} \to \mathfrak{h}$ is  diagonalizable  with an ordered basis of eigenvectors 
		$\mathcal{B} = \left\{ \bv{i}\right\}$ and corresponding eigenvalues $d_{i}$, $1 \le i \le n$.
		Then $\mathbf{D}$ is a derivation of $\mathfrak{h}$  if and only if for all basis vectors $\bv{i}$ and $\bv{j}$, we have that $\left[\bv{i}, \bv{j}\right] = 0$ or $\left[\bv{i}, \bv{j}\right]$ is an eigenvector with eigenvalue $d_{i} + d_{j}$.	
		\end{lemma} 
	As it pertains to algebraic $\tensor{T}$-solitons, there are two extreme cases concerning the eigenvalues of a diagonal derivation on a Lie algebra $\mathfrak{h}$ that we will mention briefly.
	The first is the case where $\mathbf{D} = \widehat{\tensor{T}\left[\met{g}\right]} - \kappa \mathrm{Id}$ has only one eigenvalue (i.e.,  a repeated eigenvalue of multiplicity $n = \dim \mathfrak{h}$).
	In this case,  $\tensor{T}\left[\met{g}\right]$ is a scalar multiple of the metric $\met{g}$.
	Note that a self-similar solution of the indicated form evolves by scaling only and there is no diffeomorphism action.
	Such a metric would (algebraically speaking) play the role for $\tensor{T}$ that an Einstein metric does for  $\ric$.
	
	The second case is when all eigenvalues of the diagonal derivation $\mathbf{D} = \widehat{\tensor{T}\left[\met{g}\right]} - \kappa \mathrm{Id}$ have multiplicity one.
	In this case, the Lie algebra $\mathfrak{h}$ must admit a \emph{nice basis}.
	Following \cite{LauretNice}, we say that a basis $\mathcal{B} = \left\{ \bv{i}\right\}$ for a Lie algebra $\mathfrak{h}$ is \emph{nice} if the structure constants defined by $\left[ \bv{i}, \bv{j} \right] = c_{ij}^{k} \bv{k}$ satisfy
		\begin{itemize}
			\item for all $i, j$, there exists at most one $k$ such that $c_{ij}^{k} \ne 0$, and
			
			\item for all $i, k$, there exists at most one $j$ such that $c_{ij}^{k} \ne 0$.
		\end{itemize}
	The condition on a basis of $\mathfrak{h}$ being nice can thusly be interpreted as requiring the Lie bracket of any two basis basis vectors $\bv{i}$ and $\bv{j}$  be zero or belong to the span of a third basis vector $\bv{k}$, and two nonzero brackets $\left[ \bv{i}, \bv{j} \right]$ and $\left[ \bv{p}, \bv{q}\right]$ are non-zero scalar multiples of each other if and only if $\left\{i, j\right\} = \left\{p, q\right\}$ or $\left\{i, j\right\}$ and $\left\{p, q\right\}$ are disjoint.
	It follows from  \lemref{evlem} that if  $\mathbf{D} = \widehat{\tensor{T}\left[\met{g}\right]} - \kappa \mathrm{Id}$ is a diagonal derivation with distinct eigenvalues, then the Lie algebra $\mathfrak{h}$ must admit a nice basis. 
	Note that all three-dimensional unimodular Lie groups admit a nice basis for their Lie algebras (see \secref{Milnor}).
	See \cite{LauretNice} for further discussion of Lie algebras admitting a ``nice basis'' and the relationship between a ``nice basis" and stably Ricci diagonal flows as introduced in \cite{Payne} by Payne.

	\section{Algebraic Solitons on Three-Dimensional Unimodular Lie Groups}
	\label{algsolliegroup}
	\subsection{Milnor frames}
	\label{Milnor}
	Let $\mathcal{H}$ be a three-dimensional  unimodular Lie group with Lie algebra $\mathfrak{h}$ and left invariant metric $\met{g}$.
	Throughout what follows, $\bv{1}, \bv{2}, \bv{3}$ will be a left invariant frame  with dual co-frame denoted by $\cv{1}, \cv{2},$ and $\cv{3}$.
	
	In \cite{Milnor}, Milnor establishes the existence of a left invariant  $\met{g}$-orthonormal frame $\bv{1}, \bv{2}, \bv{3}$ for $\mathfrak{h}$ such that the algebraic structure of $\mathfrak{h}$ is determined by the non-zero bracket relations
		\begin{equation}
		\label{structure}
			\left[\bv{2}, \bv{3}\right] = \lambda^{1} \bv{1} \hspace{.5in} \left[\bv{3}, \bv{1}\right] = \lambda^{2}\bv{2} \hspace{.5in} \left[\bv{1}, \bv{2}\right] =\lambda^{3} \bv{3}.
		\end{equation}
	The Levi-Civita connection of $\met{g}$ is completely determined by the Koszul formula, and with respect to the indicated frame the defining covariant derivatives are
		\begin{center}
		\begin{tabular}{l l l}
		\label{connection}
			$\cov{\bv{1}}{\bv{2}} = \mu^{1} \bv{3}$\quad\quad\quad  &$\cov{\bv{1}}{\bv{3}} =-\mu^{1} \bv{2}$ \quad\quad\quad\quad &$\cov{\bv{2}}{\bv{3}} = \mu^{2} \bv{1}$ \\
				&&\\
				$\cov{\bv{3}}{\bv{2}} = -\mu^{3}\bv{1}$ \quad\quad\quad  &$\cov{\bv{3}}{\bv{1}} = \mu^{3}\bv{2}$ \quad\quad\quad\quad  &$\cov{\bv{2}}{\bv{1}} =- \mu^{2} \bv{3}$,\\
		\end{tabular}
		\end{center}
		where $\mu^{i} = \frac{1}{2}\lp \lambda^{1} + \lambda^{2} + \lambda^{3}\rp - \lambda^{i}$.
	The principal sectional curvatures are thusly given by	
	\begin{equation}
		\label{sectional1}
			K_{l} = K\lp\bv{m} \wedge \bv{n}\rp = \lambda^{l}\mu^{l} - \mu^{m}\mu^{n}, \hspace{.5in} \textrm{$\lp l, m, n\rp$ is a permutation of $\lp 1, 2, 3 \rp$}.
		\end{equation}
	The Ricci tensor $\ric\left[\met{g}\right] = R_{ij}\cv{i} \otimes \cv{j}$ is diagonalized with respect to the indicated frame, and the  component functions take the form	
		\begin{equation}
		\label{ricci}
			R_{ll} = 2\mu^{m}\mu^{n} = K_{m} + K_{n}, \hspace{.5in}  \textrm{$(l, m, n)$ is a permutation of $(1, 2, 3)$.}
		\end{equation}
		
	Following \cite{Milnor}, we observe that if the metric $\met{g}$ is altered by declaring the basis
		\[
		\av{1} = BC\bv{1}, \hspace{.25in} \av{2} = AC \bv{2}, \hspace{.25in} \av{3} = AB\bv{3}, \hspace{.5in} A, B, C \in \mathbb{R}_{> 0}
		\]
	 to be orthonormal, then we find that the resulting structure constants are all scaled by positive constants:
	 	\[
			\left[\av{2}, \av{3}\right] = A^{2}\lambda_{1}\av{1}, \hspace{.5in} \left[\av{3}, \av{1}\right] = B^{2}\lambda_{2}\av{2}, \hspace{.5in} \left[\av{1}, \av{2}\right] = C^{2}\lambda_{3} \av{3}.
		\]
	As such, we can assume that the left invariant metric $\met{g}$ is expressed relative to  a left invariant frame $\bv{1}, \bv{2}, \bv{3}$ and its dual co-frame $\cv{1}, \cv{2}, \cv{3}$ as $\met{g} = A\, \meti{1} + B\, \meti{2} + C\, \meti{3}$, with the structure constants \eqref{structure} defining the Lie algebra satisfying $\lambda^{i} \in \left\{1, 0, -1\right\}$.
	Further, by assuming an orientation for the Lie algebra $\mathfrak{h}$, we can assume that there at least as many positive structure constants as negative structure constants, and that by appropriately ordering our basis we have $\lambda^{1} \ge \lambda^{2} \ge \lambda^{3}$.
	
	Given a left invariant metric $\met{g}$ on $\mathcal{H}$, we will refer to a left invariant frame $\bv{1}, \bv{2},$ and $\bv{3}$ that 
		\begin{enumerate}
		\item   diagonalizes the metric $\met{g}$, (i.e., $\met{g} = A\, \meti{1} + B\, \meti{2} + C\, \meti{3}$), and
		\item  diagonalizes the structure constants of $\mathfrak{h}$ as in \eqref{structure} with  $\lambda_{i} \in \left\{1, 0, -1\right\}$ and $\lambda^{1} \ge \lambda^{2} \ge \lambda^{3}$,
		\end{enumerate}
	as a \emph{Milnor frame} for $\met{g}$.
	In what follows, we will work exclusively with Milnor frames for $\met{g}$.
	
	The six possibilities for the structure constants of (oriented) Lie algebra $\mathfrak{h}$ corresponding to a simply connected three-dimensional unimodular Lie group are recorded in the table below:
		\begin{center}
		\begin{tabular}{ | c | c |}
		\hline
		$\lp \lambda^{1} , \lambda^{2}, \lambda^{3}\rp$ & Lie group \\ \hline
		$\lp 1, 1, 1 \rp$	& $\mathrm{SU}(2)$\\ \hline
		$\lp 1, 1, -1 \rp$ & $\widetilde{SL(2, \mathbb{R})}$\\ \hline
		$\lp 1, 1, 0\rp $	& $E(2$)\\ \hline
		$\lp 1, 0, -1\rp$	& $\mathrm{E}(1,1)$\\ \hline
		$\lp 1, 0, 0\rp$	& 3-dim Heisenberg group\\ \hline
		$\lp 0, 0, 0 \rp$	& $\mathbb{R}^{3}$\\ \hline
		\end{tabular}
		\end{center}
	Note that if  $\met{g}$ is a left invariant metric on $\mathcal{H}$ expressed in a Milnor frame as $\met{g} = A \meti{1} + B\meti{2} + C\meti{3}$, then the principal sectional curvatures are given by
	 \begin{equation}
		\label{sectional}
			K_{l} = K\lp\bv{m} \wedge \bv{n}\rp = \tilde{\lambda}^{l}\tilde{\mu}^{l} - \tilde{\mu}^{m}\tilde{\mu}^{n}, \hspace{.5in} \textrm{$\lp l, m, n\rp$ is a permutation of $\lp 1, 2, 3 \rp$},
		\end{equation}
	where $\tilde{\lambda}^{1} = \sqrt{\frac{A}{BC}}\lambda^{1}$, 
	$\tilde{\lambda}^{2} = \sqrt{\frac{B}{CA}}\lambda^{2}$, $\tilde{\lambda}^{3} = \sqrt{\frac{A}{BC}}\lambda^{3}$, and 
	$\tilde{\mu}^{i} = \frac{1}{2}\lp \tilde{\lambda}^{1} + \tilde{\lambda}^{2} + \tilde{\lambda}^{3}\rp  - \tilde{\lambda}^{i}$,  $1 \le i \le 3$.
	Furthermore,  the components of the Ricci tensor $\ric\left[\met{g}\right] = R_{11} \meti{1} + R_{22} \meti{2} + R_{33} \meti{3}$ are given by
		\begin{equation}
		\label{ricciorthogonal}
			R_{11} = A\lp K_{2} + K_{3}\rp, \hspace{.5in} R_{22} = B\lp K_{1} + K_{3} \rp , \hspace{.5in} R_{33} = C\lp K_{1} + K_{2} \rp,
		\end{equation}
	 and the tensor $\tensor{Rm}^{2}\left[\met{g}\right] = \check{R}_{ij}\tp{i}{j}$ is diagonalized with components functions given  by
		\begin{equation}
		\label{rm2generic}
		\begin{cases}
		\medskip
		\check{R}_{11} &= 2A\left(\left(K_{2}\right)^2 + \left(K_{3}\right)^2\right) \\ 
		\medskip
		\check{R}_{22} &= 2B\left(\left(K_{3}\right)^2 + \left(K_{1}\right)^2\right) \\
		\check{R}_{33} &= 2C\left(\left(K_{3}\right)^2 + \left(K_{1}\right)^2\right).
		\end{cases}
		\end{equation}
	It follows that cross curvature tensor $\ccf\left[\met{g}\right] = H_{ij} \tp{i}{j}$,  $\mathrm{RG}$-$\mathrm{2}$ tensor $\rg\left[\met{g}\right] = RG_{ij}\tp{i}{j}$ are all diagonalized with respect to a Milnor frame.
	The components  are, respectively, 
	
		\begin{equation}
		\label{xcfgeneric}
		H_{11} = AK_{2}K_{3}, \quad H_{22} = BK_{1}K_{3}, \quad \textrm{and } \quad H_{33} = CK_{1}K_{2},
		\end{equation}
		and 
		\begin{equation}
		\label{rg2generic}
		\begin{cases}
		\medskip
		RG_{11} &= -A\left(2\left(K_{2} + K_{3}\right) + \alpha \left(\left(K_{2}\right)^2 + \left(K_{3}\right)^2\right)\right)\\
		\medskip
		RG_{22} &= -B\left(2\left(K_{3} + K_{1}\right) + \alpha \left(\left(K_{3}\right)^2 + \left(K_{1}\right)^2\right)\right)\\ 
		RG_{33} &= -C\left(2\left(K_{1} + K_{2}\right) + \alpha \left(\left(K_{1}\right)^2 + \left(K_{2}\right)^2\right)\right).
		\end{cases}
		\end{equation}
		
	Finally, we note that the corresponding operators $\ccfop\left[\met{g}\right]$ and  $\rgop\left[\met{g}\right]$ are diagonalized with respect to the indicated frame and take the form
		\begin{align}
		\label{operatorgeneric}
		\ccfop\left[\met{g}\right] &= \diag\left(K_{2}K_{3}\,,\, K_{3}K_{1}\,,\, K_{1}K_{2}\right)  \\
		 \rgop\left[\met{g}\right] &= \diag \left(\frac{RG_{11}}{A}, \frac{RG_{22}}{B}, \frac{RG_{33}}{C}\right) \nonumber.
		\end{align}
	
	We will now classify the algebraic solitons on three-dimensional unimodular Lie groups that give rise to self-similar solutions of the $\mathrm{XCF}$ and $\mathrm{RG}$-$\mathrm{2}$ flows.
	Before proceeding, recall from \propref{solitonselfsimilar} and \propref{algimpliessol} that all algebraic $\mathrm{XCF}$-solitons give rise to self-similar solutions of the $\mathrm{XCF}$ flow, whereas only steady algebraic $\mathrm{RG}$-$\mathrm{2}$-solitons give rise to self-similar solutions of the $\mathrm{RG}$-$\mathrm{2}$ flow.
	 
	\subsection{$\mathbb{R}^{3}$}
	\label{r3}
	
	We begin with the trivial case of $\mathcal{H} = \mathbb{R}^{3}$. 
	Since all left invariant metrics on the Abelian group $\mathbb{R}^{3}$ are flat, then it follows immediately that all left invariant metrics $\met{g}$ are fixed points for the $\mathrm{XCF}$ and $\mathrm{RG}$-$\mathrm{2}$ flow.
	Such metrics can be regarded as trivial solitons.

	\subsection{Heisenberg Group}
		\label{Heis}
		Let $\met{g}$ be a left invariant metric for the three-dimensional Heisenberg group and let $\bv{1}, \bv{2}, \bv{3}$ be a Milnor frame such that 
		$\met{g}= A \meti{1} + B \meti{2} + C \meti{3}$ and the Lie algebra structure is determined by the non-zero bracket  $\left[ \bv{2}, \bv{3}\right] = \bv{1}$.
		By using an automorphism of the Lie algebra, one can further assume that $ B = C = 1$ and that the metric takes the form $\met{g}	= A \meti{1} +  \meti{2} +  \meti{3}$.
		
		From \eqref{sectional} we find that  the principal sectional curvatures are
			\begin{equation}
			\label{sectionalheis}
			K_{1} = K\left(\bv{2} \wedge \bv{3}\right) = - \frac{3}{4} A, \hspace{.25in}
			K_{2}  = K\left(\bv{1} \wedge \bv{3}\right) = \frac{1}{4} A, \hspace{.25in}
			K_{3} = K\left(\bv{1} \wedge \bv{2}\right) = \frac{1}{4}A.
			\end{equation}
		
		We will now make use of the algebraic structure to give a simple criterion for when a left invariant metric $\met{g}$ on the Heisenberg group is an algebraic $\tensor{T}$-soliton for the evolution equation $\frac{\partial \met{g}}{\partial t} = \tensor{T}\left[\met{g}\right]$.
		
		\begin{lemma}
		\label{heislemma}
		Let $\met{g} = A \meti{1} + \meti{2} + \meti{3}$ be a left invariant metric (expressed relative to a Milnor frame) on the three-dimensional Heisenberg group.
		If $\widehat{\tensor{T}}\left[\met{g}\right]$ is diagonalized relative the Milnor frame, then $\met{g}$ is an algebraic $\tensor{T}$-soliton for the evolution equation 
		$\frac{\partial \met{g}}{\partial t} = \tensor{T}\left[\met{g}\right]$ with soliton constant $\kappa = \widehat{T}^{2}_{2} + \widehat{T}^{3}_{3} - \widehat{T}^{1}_{1}$.
		\end{lemma}
		
		\begin{proof}
		Let $\tensor{D}\left[\met{g}\right] = \widehat{\tensor{T}}\left[\met{g}\right] - \kappa \mathrm{Id}$ and assume that 
		$\widehat{\tensor{T}}\left[\met{g}\right] : \mathfrak{h} \to \mathfrak{h}$ is diagonalized relative to the Milnor frame.
		Owing to the Lie bracket structure of the Lie algebra $\mathfrak{h}$, it follows from \lemref{evlem} that the operator $\tensor{D} = \lp d^{i}_{j} \rp = \lp \widehat{T}^{i}_{j}- \kappa \delta^{i}_{j}\rp$ is a derivation of $\mathfrak{h}$ if and only if $\left[ \bv{2}, \bv{3} \right] = \bv{1}$ is an eigenvector for $\tensor{D}$ with eigenvalue $d^{1}_{1} = d^{2}_{2} + d^{3}_{3}$.
		Accordingly, we must have $\widehat{T}^{1}_{1} - \kappa = \lp \widehat{T}^{2}_{2} -\kappa \rp +   \lp \widehat{T}^{3}_{3} -\kappa \rp$, or equivalently, $\kappa  =  \widehat{T}^{2}_{2} + \widehat{T}^{3}_{3} - \widehat{T}^{1}_{1}$.
		This completes the proof.
		\end{proof}
		
		The classification of algebraic $\mathrm{XCF}$ and $\mathrm{RG}$ solitons on the Heisenberg group that give rise to self-similar solutions of their respective flows now follows easily.
		\subsubsection{Algebraic XCF-solitons} 
		From \eqref{xcfgeneric} and \eqref{sectionalheis}, we find that the cross curvature tensor is $\ccf\left[\met{g}\right]  = \frac{A^{3}}{16}  \tp{1}{1}  -\frac{3A^{2}}{16} \tp{2}{2} -\frac{3A^{2}}{16}\tp{3}{3}$, and the corresponding cross curvature operator $\ccfop\left[\met{g}\right]$ is represented in the given frame by
			\begin{equation}
				\label{Heisxcfop}
				\ccfop\left[\met{g}\right] = \diag\left(\frac{A^{2}}{16}, -\frac{3 A^2}{16}, -\frac{3 A^2}{16}\right).
			\end{equation}

		\begin{theorem}
		For all choices of $A$, the left invariant metric $\met{g} = A \meti{1} +  \meti{2} +  \meti{3}$ is a shrinking algebraic $\mathrm{XCF}$ soliton for the $\mathrm{+XCF}$ on the Heisenberg group. (And likewise, an expanding algebraic $\mathrm{-XCF}$ soliton.)  
		\end{theorem}
		
			\begin{proof}
			Let $\mathbf{D}\left[\met{g}\right] =  \ccfop\left[\met{g}\right] - \kappa \mathrm{Id}$, where $\ccfop = \ccfop\left[\met{g}\right]$ is as in \eqref{Heisxcfop} and $\kappa \in \mathbb{R}$.
			 The operator $\mathbf{D} = \lp d^{i}_{j}\rp = \lp \widehat{H}^{i}_{j} - \kappa \delta^{i}_{j}\rp$ is diagonalized with respect to the given frame, and applying \lemref{heislemma} we find that $\met{g}$ is an algebraic soliton with soliton constant  $\kappa = \widehat{H}^{2}_{2} + \widehat{H}^{3}_{3} - \widehat{H}^{1}_{1} = - \frac{7}{16}A^2$, which completes the proof.
			\end{proof}

		\subsubsection{Algebraic $\mathrm{RG}$-$\mathrm{2}$ Solitons}
		\label{rg2solheis}
			
			Combining \eqref{rg2generic} and \eqref{sectionalheis}, the $\mathrm{RG}$-$\mathrm{2}$ tensor  of the metric $\met{g}$ is 					
					
					\begin{equation}
					\rg\left[\met{g}\right]=  -\left(\frac{1}{8} \alpha A^{3} + A^{2}\right)\meti{1}  +\lp A - \frac{5}{8}\alpha A^{2} \rp \meti{2}  +\lp A - \frac{5}{8} \alpha A^{2}  \rp\meti{3},
					\end{equation}
			and the corresponding RG-2 operator is $\widehat{\rg}\left[\met{g}\right] = \diag\left( - \frac{\alpha}{8}A^2 - A, A  - \frac{5}{8}\alpha A^{2}, A -\frac{5}{8} \alpha A^{2} \right)$. 
			We immediately establish the following.
				\begin{theorem}
					\label{rg2heissol}
				The left invariant metric $\met{g} = A \meti{1} + \meti{2} + \meti{3}$ is a steady algebraic soliton that gives rise to a self-similar solution of the $\mathrm{RG}$-$\mathrm{2}$ flow 
				$ \frac{\partial \met{g}}{\partial t} = -2 \ric\left[\met{g}\right] - \frac{4}{3A}\rg\left[\met{g}\right]$ (i.e., when $\alpha = \frac{8}{3A}$).
				\end{theorem}			
				
				\begin{proof}
					Since we are focused on finding self-similar solutions to the $\mathrm{RG}$-$\mathrm{2}$ flow on the Heisenberg group, we will only be looking for steady algebraic $\mathrm{RG}$-2 solitons of the flow.
					The operator $\tensor{D} = \rgop\left[\met{g}\right] - \kappa \mathrm{Id}$ is diagonalized with respect to the Milnor frame.
					Setting $\kappa = 0$ and applying \lemref{heislemma}, we find that $\tensor{D}$ is a derivation of the Lie algebra $\mathfrak{h}$ if and only if 
					$\widehat{RG}_{1}^{1} =  \widehat{RG}_{2}^{2}  + \widehat{RG}_{3}^{3}$.
					The equation  $\widehat{RG}_{1}^{1} =  \widehat{RG}_{2}^{2}  + \widehat{RG}_{3}^{3}$ is satisfied if and only if
						$- \frac{\alpha}{8}A^2 - A = 2\lp A  - \frac{5}{8}\alpha A^{2}\rp$,
					or $3A\lp \frac{3}{8}\alpha A -  1\rp = 0$.
					Since $A = \met{g}\lp \bv{1}, \bv{1}\rp$, the result follows.
				\end{proof}
				
	\subsection{$\mathrm{E}(2)$}
	\label{euclideanmotions}
			\label{xcfeuclid}
		Let $\met{g}$ be a left invariant metric for $\mathrm{E}(2)$,  the three-dimensional group of (orientation preserving) isometries of the Euclidean plane, and let $\bv{1}, \bv{2}, \bv{3}$ be a Milnor frame for $\met{g}$ such that $\met{g} = A \meti{1} + B \meti{2} + C \meti{3}$, with the structure of the Lie algebra  determined by the non-zero brackets
				\begin{equation*}
				\left[ \bv{2}, \bv{3}\right] =\bv{1} \hspace{.25in} \textrm{and} \hspace{.25in}
				\left[ \bv{3}, \bv{1}\right] = \bv{2}.
				\end{equation*}
		Using an automorphism of the Lie algebra, we can further assume that $A = 1$ and write  $\met{g} =   \meti{1} + B \meti{2} + C \meti{3}$. \\
		According to \eqref{sectional}, the principal sectional curvatures are
			\begin{equation}
			\label{sectionale2}
			K_{1}  = \frac{\left(B + 3\right)\left(B - 1\right)}{4BC}, \quad\quad
			K_{2}  = \frac{\left(3B + 1\right)\left(1- B\right)}{4BC}, \quad\quad
			K_{3} =  \frac{\left(B - 1\right)^2}{4BC}.
			\end{equation}
			Note that when $B = 1$, we obtain family of  flat metrics on $\mathrm{E}(2)$.
		
		As before, we can make use of the algebraic structure of the Lie algebra to provide a simple criterion for when a left invariant metric is an algebraic $\tensor{T}$-soliton for a geometric evolution equation of the form $\frac{\partial \met{g}}{\partial t} = \tensor{T}\left[\met{g}\right]$.
		\begin{lemma}
		\label{E2lemma}
		Let $\met{g} =  \meti{1} + B\meti{2} + C\meti{3}$ be a left invariant metric (expressed relative to a Milnor frame) on $\mathrm{E}(2)$.
		If $\widehat{\tensor{T}}\left[\met{g}\right]$ is diagonalized relative to the Milnor frame, then $\met{g}$ is an algebraic $\tensor{T}$-soliton for the evolution equation 
		$\frac{\partial \met{g}}{\partial t} = \tensor{T}\left[\met{g}\right]$ if and only if $\widehat{T}^{1}_{1} = \widehat{T}^{2}_{2}$.
		Moreover, the soliton constant is $\kappa = \widehat{T}^{3}_{3}$.
		\end{lemma}
		
		\begin{proof}
		The proof proceeds in an identical manner to that of the proof of  \lemref{heislemma}.
		Let $\tensor{D}\left[\met{g}\right] = \widehat{\tensor{T}}\left[\met{g}\right] - \kappa \mathrm{Id}$ and assume that 
		$\widehat{\tensor{T}}\left[\met{g}\right] : \mathfrak{h} \to \mathfrak{h}$ is diagonalized relative to the Milnor frame.
		Owing to the Lie bracket structure of the Lie algebra $\mathfrak{h}$, it follows from \lemref{evlem} that the operator $\tensor{D} = \lp d^{i}_{j} \rp = \lp \widehat{T}^{i}_{j}- \kappa \delta^{i}_{j}\rp$ is a derivation of $\mathfrak{h}$ if and only if $\left[ \bv{2}, \bv{3} \right] = \bv{1}$ and $\left[\bv{3}, \bv{1}\right] = \bv{2}$ are eigenvectors for $\tensor{D}$ with eigenvalues $d^{1}_{1} = d^{2}_{2} + d^{3}_{3}$ and $d^{2}_{2} = d^{3}_{3} + d^{1}_{1}$.
		Accordingly, we must have  $d^{3}_{3} = \widehat{T}^{3}_{3} - \kappa = 0$ and $\widehat{T}^{1}_{1} =  \widehat{T}^{2}_{2}$, which completes the proof.
		\end{proof}
			\subsubsection{$\mathrm{XCF}$-solitons}
			It follows from \eqref{xcfgeneric} and  \eqref{sectionale2} that the cross curvature tensor  is 				
				\begin{equation}
				\label{xcfE2}
					\ccf\left[\met{g}\right] = 
					- \frac{\left(3B + 1\right)\left(B - 1\right)^3}{16 B^{2} C^{2}} \meti{1} +  
					\frac{\left(B + 3\right)\left(B - 1\right)^3}{16 B C^{2}}\meti{2} -
					\frac{\lp B + 3\rp \lp B - 1\rp^2\lp 3B + 1\rp}{16 B^2 C} \meti{3}.
				\end{equation}
			The non-zero components of $\ccfop\left[\met{g}\right]$ with respect to the indicated frame are $ \widehat{H}^{1}_{1} = H_{11}$, $\widehat{H}^{2}_{2} = \frac{H_{22}}{B}$, and $\widehat{H}^{3}_{3} = \frac{H_{33}}{C}$,  and the cross curvature operator is
				\begin{equation*}
				 \ccfop\left[\met{g}\right] = \diag \lp - \frac{\left(3B + 1\right)\left(B - 1\right)^3}{16 B^{2} C^{2}},  
				 \frac{\left(B + 3\right)\left(B - 1\right)^3}{16B^2 C^{2}}, 
				 \frac{\lp B + 3\rp \lp B - 1\rp^2\lp 3B + 1\rp}{16 B^2 C^{2}} \rp.
				\end{equation*}
			\begin{theorem}
			All flat metrics on $\mathrm{E}(2)$ are fixed points for the $\mathrm{XCF}$ on $\mathrm{E}(2)$ (and are thus (trivial) steady algebraic solitons).
			These are the only algebraic solitons for the $\mathrm{XCF}$ on $\mathrm{E}(2)$.
			\end{theorem}
				\begin{proof}
				That the flat metrics on $\mathrm{E}(2)$ are fixed points for the $\mathrm{XCF}$ is clear from the definition of the cross curvature tensor.
				To see that these are the only algebraic solitons for the $XCF$ on $\mathrm{E}(2)$, we will employ \lemref{E2lemma}. 
				By \lemref{E2lemma},  $\widehat{H}^{1}_{1} = \widehat{H}^{2}_{2}$ and $\kappa = \widehat{H}^{3}_{3}$.
				Owing to \eqref{operatorgeneric}, we see that $\widehat{H}^{1}_{1} = \widehat{H}^{2}_{2}$ if and only if $K_{2}K_{3} = K_{1}K_{3}$.
				Thus we must have $K_{3} = 0$ or $K_1 = K_2$.
				From \eqref{sectionale2}, we see that $K_{3} = 0 $ if and only if $ B = 1$ and that $K_{1} = K_{2}$ if and only if $B = 1$ or $B = -\frac{1}{2}$.
				Since $B = \met{g}\lp \bv{2}, \bv{2}\rp$, we conclude that $\met{g}$ is an algebraic soliton if and only if $B = 1$.
				It follows that such a metric is flat, and is thus a fixed point of the flow.
				\end{proof}

			\subsubsection{$\mathrm{RG}$-$\mathrm{2}$ solitons}
			\label{rg2euclid}
			The RG-2 tensor of $\met{g}$ is  $\rg\left[\met{g}\right] = RG_{11}\meti{1} + RG_{22} \meti{2} + RG_{33}\meti{3}$, where	
				\begin{align}
				\label{rg2compE2}
					\medskip
					RG_{11} 	&= - \frac{ \lp B  -  1 \rp \lp 5 \alpha B^{3} -3\alpha B^2 - 8 B^2 C - B\alpha  - 8 BC - \alpha \rp}{8  B^2C^2} \notag \\
					\medskip
					RG_{22}	&= - \frac{ \lp B - 1\rp  \lp  \alpha B^{3} + \alpha B^2 + 8B^2 C + 3 B \alpha + 8 BC - 5\alpha \rp}{8  BC^2} \notag	\\			
					RG_{33}	&= - \frac{ \lp B - 1 \rp^2 \lp 5 \alpha B^2  + 6 \alpha B -  8 B C + 5 \alpha \rp}{8 B^2 C},
				\end{align}
			and the corresponding $\mathrm{RG}$-$\mathrm{2}$ operator is $\rgop = \diag \lp \widehat{RG}^{1}_{1}, \widehat{RG}^{2}_{2}, \widehat{RG}^{3}_{3}\rp $, where
				\begin{equation}
				\label{rg2opE2}
				\widehat{RG}^{1}_{1} = RG_{11} \quad\quad \widehat{RG}^{2}_{2} = \frac{RG_{22}}{B}, \quad\quad \widehat{RG}^{3}_{3} = \frac{RG_{33}}{C}.
				\end{equation} 
			\begin{theorem}
			All flat metrics on $\mathrm{E}(2)$ are fixed points for the $\mathrm{RG}$-$\mathrm{2}$ (and are thus trivial steady algebraic solitons).  Furthermore, the flat metrics on $\mathrm{E}(2)$ are the only left invariant steady algebraic solitons for the $\mathrm{RG}$-$\mathrm{2}$ flow on $\mathrm{E}(2)$.
			\end{theorem}			
			
			\begin{proof}
			Owing to the definition of the $\mathrm{RG}$-$\mathrm{2}$ tensor, it is clear that the flat metrics are fixed points (and thus trivial algebraic solitons for the $\mathrm{RG}$-$\mathrm{2}$ flow).
			To see that the flat metrics are the only left invariant steady algebraic solitons for the $\mathrm{RG}$-$\mathrm{2}$ flow on $\mathrm{E}(2)$,  note by \lemref{E2lemma} that in order for $\rgop\left[\met{g}\right]$ to be a steady algebraic we must have $\widehat{RG}_{33} = 0$ and  $\widehat{RG}_{11} = \widehat{RG}_{22}$.
			
			Combining \eqref{rg2generic} and \eqref{operatorgeneric}, we see that $\widehat{RG}_{11} = \widehat{RG}_{22}$ if and only if 
				\[
				K_{2}\lp2 + \alpha K_{2} \rp = K_{1}\lp 2 + \alpha K_{1}\rp.
				\]
			Substituting \eqref{sectionale2} into the above and making using of the fact that we can assume that $B \ne 1$, this is equivalent to the equation
				\[
				-\lp 3B + 1 \rp \lp 2+ \alpha \frac{\left(3B + 1\right)\left(1- B\right)}{4BC}\rp = \lp B + 3 \rp \lp 2 + \alpha \frac{\left(B + 3\right)\left(B - 1\right)}{4BC}\rp.
				\]
			The above equation admits solutions when $B = -1$ and when $C = \frac{\alpha \lp B - 1\rp^{2}}{4B}$. 
			Since $B $ must be greater than zero, we only concern ourselves with the later.
			Substituting $C =  \frac{\alpha \lp B - 1\rp^{2}}{4B}$ into $\widehat{RG}_{33}$ as in \eqref{rg2opE2}, we find that 
			$\widehat{RG}_{33} = 0$ if and only if $-\frac{1}{2} \frac{\lp 3 B  + 1 \rp \lp B + 3 \rp}{B} = 0$.
			Again, since $B$ must be greater than zero, we conclude that the only left invariant steady algebraic $\mathrm{RG}$-$\mathrm{2}$ solitons on $\mathrm{E}(2)$ are the flat metrics.
			
			\end{proof}

	\subsection{$\mathrm{E}(1,1)$}
	\label{LMsolvable}
		Let $\met{g}$ be a left invariant metric for $\mathrm{E}(1,1)$,  the solvable three-dimensional group of isometries of the standard Lorentz-Minkowski plane, expressed relative to a Milnor frame $\bv{1}, \bv{2}, \bv{3}$ as $\met{g} = A \meti{1} + B \meti{2} + C \meti{3}$,
		with the corresponding Lie algebra structure generated by the non-zero brackets
			\[
				\left[ \bv{2}, \bv{3} \right] = \bv{1}, \hspace{.5in} \left[ \bv{1}, \bv{2}\right] = - \bv{3}.
			\]
		Using an automorphism of the Lie algebra, we can further require that $C = 1$ and that the metric takes the form $\met{g} = A \meti{1} + B \meti{2} +  \meti{3}.$
		
		Making the appropriate substitutions into \eqref{sectional}, we find that the principal sectional curvatures of $\met{g}$ are 
			\begin{align}
			\label{sectionalsol}
			K_{1} &= K\left(\bv{2} \wedge \bv{3}\right) =- \frac{\lp A + 1\rp \lp 3A - 1 \rp}{4A B} \nonumber \\
			K_{2} & = K\left(\bv{1} \wedge \bv{3}\right) = \frac{\lp A + 1 \rp^2}{4 A B} \\
			K_{3} &= K\left(\bv{1} \wedge \bv{2}\right) = \frac{\lp A + 1\rp \lp A - 3\rp}{4 A B}.\nonumber
			\end{align}
		We now state a lemma for $\mathrm{E}(1,1)$ that is equivalent to  \lemref{E2lemma} for $\mathrm{E}(2)$.
			\begin{lemma}
			\label{Sollemma}
			Let $\met{g} =  A\meti{1} + B\meti{2} + \meti{3}$ be a left invariant metric (expressed relative to a Milnor frame) on $\mathrm{E}(1,1)$.
		If $\widehat{\tensor{T}}\left[\met{g}\right]$ is diagonalized relative the Milnor frame, then $\met{g}$ is an algebraic $\tensor{T}$-soliton for the evolution equation 
		$\frac{\partial \met{g}}{\partial t} = \tensor{T}\left[\met{g}\right]$ if and only if $\widehat{T}^{1}_{1} = \widehat{T}^{3}_{3}$.
		Moreover, the soliton constant is $\kappa = \widehat{T}^{2}_{2}$.
			\end{lemma}
			\begin{proof}
			The proof proceeds in an identical fashion to the proof of \lemref{E2lemma} and the details are left to the reader.
			\end{proof}
		\subsubsection{XCF solitons}
		Making the appropriate substitutions into \eqref{xcfgeneric}, the components of the cross curvature tensor  $\ccf\left[ \met{g}\right] = H_{11}\tp{1}{1} + H_{22}\tp{2}{2} + H_{33} \tp{3}{3}$ are found to be
			\begin{equation}
			\label{xcfopsol}
				H_{11} = \frac{\lp A + 1\rp^3 \lp A - 3\rp}{16 A B^{2}} \quad
				H_{22} =  \frac{ \lp A + 1\rp^2 \lp1 - 3 A\rp \lp A - 3\rp}{16 A^2 B} \quad
				H_{33} =\frac{\lp A + 1 \rp ^3 \lp 1 - 3A\rp}{16A^{2} B^{2}},
			\end{equation}
		and the cross curvature operator $\ccfop$ is represented in the given frame by
		$\ccfop\left[\met{g}\right] = \diag\left(\widehat{H}^{1}_{1}, \widehat{H}^{2}_{2}, \widehat{H}^{3}_{3}\right),$
		where $\widehat{H}^{1}_{1} = \frac{H_{11}}{A}, \widehat{H}^{2}_{2} = \frac{H_{22}}{B},$ and $\widehat{H}^{3}_{3} = H_{33}$.
			\begin{theorem}
			The left invariant metric $\met{g} = A\meti{1} + B \meti{2} + \meti{3}$ is an expanding  (res. shrinking) algebraic soliton for the $\mathrm{+XCF}$  (res. $\mathrm{-XCF}$) if and only if $A = 1$ and the  soliton constant is $\kappa = \frac{1}{B^2}$ (res. $\kappa = - \frac{1}{B^2}$).
			\end{theorem}
			
			\begin{proof}
				Since $\mathbf{D} = \ccfop\left[\met{g}\right] - \kappa \mathrm{Id}$ is diagonalized relative to the Milnor frame, it follows from \lemref{Sollemma} that $\mathbf{D} = \lp d^{i}_{j}\rp = \lp \widehat{H}^{i}_{j} - \kappa \delta^{i}_{j}\rp$ is a derivation of $\mathfrak{h}$ if and only if $\widehat{H}^{1}_{1} = \widehat{H}^{3}_{3}$ and $\kappa = \widehat{H}^{3}_{3}$.
					Owing to \eqref{operatorgeneric} and \eqref{xcfopsol}, it follows that we must have
					\[
					\kappa = \widehat{H}^{2}_{2} =   \frac{ \lp A + 1\rp^2 \lp1 - 3 A\rp \lp A - 3\rp}{16 A^2 B^2}, 
					\]
				and \[
					\widehat{H}^{1}_{1} = \widehat{H}^{3}_{3} \iff K_{2}K_{3} = K_{1}K_{2}.
					\]
				Thus, we must have $K_2 = 0$ (which only happens if $A = -1$) or $K_{1} = K_{3}$.
				Since $A > 0$, then according to \eqref{sectionalsol}, we see that $K_{1} = K_{3}$ if and only if $1 - 3A = A - 3$.
				We thus find that we have an algebraic soliton which gives rise to a self-similar solution when $A = 1$.
				Substituting $A = 1$ into $\kappa = \widehat{H}^{2}_{2}$ we find $\kappa =  \frac{1}{B^2}$.
			\end{proof}
		
		\subsubsection{$\mathrm{RG}$-$\mathrm{2}$ Solitons}
		\label{rg2minkowski}
			From \eqref{rg2generic} and \eqref{operatorgeneric}, respectively, the RG-2 tensor of $\met{g}$ is  $\rg\left[\met{g}\right] = RG_{11}\meti{1} + RG_{22} \meti{2} + RG_{33}\meti{3}$, where				
				\begin{align}
				\label{rg2compmink}
					RG_{11} 	&= - \frac{ \lp A + 1 \rp \lp \alpha A^{3} + 8A^2 B - \alpha A^2 + 3 \alpha A - 8 AB + 5 A\rp}{8 A B^2} \notag \\
					RG_{22}	&= - \frac{ \lp A + 1\rp ^2 \lp 5 \alpha A^2 - 6 \alpha A - 8 A B + 5 \alpha\rp}{8 A^2 B} \notag	\\				
					RG_{33}	&= - \frac{ \lp A + 1 \rp \lp 5 \alpha A^3 + 3 \alpha A^2 - 8 A^2 B - \alpha A + 8 A B + \alpha\rp}{8 A^2 B^2},
				\end{align}
			and the corresponding $\mathrm{RG}$-$\mathrm{2}$ operator is $\rgop = \diag \lp \widehat{RG}^{1}_{1}, \widehat{RG}^{2}_{2}, \widehat{RG}^{3}_{3}\rp $, where
				\begin{equation}
				\label{rg2opmink}
				\widehat{RG}^{1}_{1} = \frac{RG_{11}}{A}, \quad\quad \widehat{RG}^{2}_{2} = \frac{RG_{22}}{B}, \quad\quad \widehat{RG}^{3}_{3} = RG_{33}.
				\end{equation} 
			\begin{theorem}
			\label{rg2sol}
			The left invariant metric $\met{g} = A \meti{1} + B \meti{2} + \meti{3}$ is a steady algebraic soliton that gives rise to a self-similar solution of the $\mathrm{RG}$-$\mathrm{2}$ flow on $\mathrm{E}(1,1)$
			if and only if 
				\begin{enumerate}
					\item $ A = 1$ and $\alpha = 2B $, or 
					\item $A = \frac{1}{3}$ or $A = 3$ and $\alpha = \frac{3}{4}B$.
				\end{enumerate}
			\end{theorem}
			
			\begin{proof}
			Owing to \lemref{Sollemma}, we see that $ \tensor{D} = \widehat{\rg}\left[\met{g}\right]$ is a derivation of the Lie algebra $\mathfrak{h}$ if and only if 
			$\widehat{RG}^{2}_{2} = 0$ and $\widehat{RG}^{1}_{1} = \widehat{RG}^{3}_{3}$. 	
			
			From \eqref{rg2compmink} and  \eqref{rg2opmink}, we see that  $\widehat{RG}^{1}_{1}  =  \widehat{RG}^{3}_{3} $ if and only if 
				\[
				\frac{1}{2}\frac{\lp A + 1 \rp \lp A - 1 \rp \lp \alpha A^2 + 2 \alpha A - 4 AB + \alpha\rp}{A^2 B^2} = 0.
				\]
				
			If $A = 1$,  then  $\widehat{\rg}_{22} = 0$ if and only if $- \frac{1}{2}\frac{4 \alpha - 8 B}{B^2} = 0$,which establishes \emph{1.}  
			If  $ \alpha A^2 + 2 \alpha A - 4 AB + \alpha = 0$, then solving for $\alpha = \frac{4AB}{\lp A + 1\rp^2}$ and substituting into $\widehat{RG}^{2}_{2}$, we find that 
			$\widehat{RG}^{2}_{2} = 0$ if and only if 
				\[
				-\frac{1}{2} \frac{\lp 3A - 1\rp \lp A - 3\rp}{AB} = 0,
				\]
			which establishes \emph{2.}
			\end{proof}
			
			\begin{remark}
			The left invariant metrics $\met{g} = 3 \meti{1} + B \meti{2} +  \meti{3}$ and $\met{g} = \frac{1}{3} \meti{1} + B \meti{2} + \meti{3}$ are isometric under an  automorphism of the Lie algebra of $E(1, 1)$.
			With respect to the Milnor Frame $\bv{1}, \bv{2}, \bv{3}$, it is readily established that any invertible linear operator that can be represented in matrix form relative to the chosen frame  as 
			$ F = 
			\begin{pmatrix} 
			f^{1}_{1} & f^{1}_{2} & f^{1}_{3}\\ 
			0 & \pm 1 & 0 \\ 
			f^{1}_{3} & f^{3}_{2} & f^{1}_{1} \\ 
			\end{pmatrix}$ 
			is an automorphism of the Lie algebra.
			Taking 
			$F =  
			\begin{pmatrix}
			0& 0  & \sqrt{3} \\ 
			0 & - 1 & 0 \\ 
			\sqrt{3} & 0 & 0 \\ 
			\end{pmatrix}$ 
			establishes the isometry between the two metrics in question.
			 Similarly, one is able to show that every left invariant metric on $E(1, 1)$ is equivalent to one of the form $\met{g} = A\meti{1} + B \meti{2} + \meti{3}$ with $A \ge 1$.
			\end{remark}
			
		\subsection{$\widetilde{\mathrm{SL}(2)}$}
			Let $\mathcal{H} = \widetilde{SL(2)}$, the universal cover of $SL\lp 2, \mathbb{R}\rp$, and let $\met{g}$ be a left invariant metric on $\mathcal{H}$ with Milnor frame
			$\bv{1}, \bv{2}, \bv{3}$.  
			The metric then takes the form $\met{g} = A \meti{1} + B \meti{2} + C \meti{3}$ and the Lie algebra structure is determined by the non-zero brackets 			
			\[
				\left[ \bv{2}, \bv{3}\right] =\bv{1} , \hspace{.5in} 
				\left[\bv{3}, \bv{1} \right] = \bv{2}, \hspace{.5in}
				\left[ \bv{1}, \bv{2}\right] = - \bv{3}.
			\]

			We will now show that $\widetilde{\mathrm{SL}(2)}$ does not support either a left invariant algebraic $\mathrm{XCF}$  soliton or a left invariant algebraic $\mathrm{RG}$-$\mathrm{2}$ soliton that gives rise to a self-similar solution of the corresponding flow.  
			The proof relies on a lemma appropriately adapted from \lemref{evlem} and on the possible signatures that the Ricci tensor of a left invariant metric on $\widetilde{\mathrm{SL}(2)}$ may have.
			In \cite{Milnor} (Corollary 4.7), Milnor establishes that the signature of the Ricci tensor of any left invariant metric on $\widetilde{\mathrm{SL}(2)}$ must be either $\lp +, -, -\rp$ or $\lp 0, 0, -\rp$.  
			
			\begin{lemma}
			\label{SL2lemma}
			Let $\met{g} =  A\meti{1} + B\meti{2} + C\meti{3}$ be a left invariant metric (expressed relative to a Milnor frame) on $\widetilde{\mathrm{SL}(2)}$.
		If $\widehat{\tensor{T}}\left[\met{g}\right]$ is diagonalized relative the Milnor frame, then $\met{g}$ is an algebraic $\tensor{T}$-soliton for the evolution equation 
		$\frac{\partial \met{g}}{\partial t} = \tensor{T}\left[\met{g}\right]$ if and only if $\widehat{T}^{1}_{1} = \widehat{T}^{2}_{2} = \widehat{T}^{3}_{3}$ and $\tensor{T}\left[\met{g}\right]$ is a scalar multiple of $\met{g}$.
			\end{lemma}
			
			\begin{proof}
			Assume that $\widehat{\tensor{T}}\left[\met{g}\right]$ is diagonalized with respect to the selected frame.
			Since $\left[\bv{l}, \bv{m}\right] \ne 0$ for all pairs $(l, m)$ where $l$ and $m$ are distinct, then according to \lemref{evlem} the operator $\tensor{D}\left[\met{g}\right] = \widehat{\tensor{T}}\left[\met{g}\right] - \kappa \mathrm{Id} : \mathfrak{h} \to \mathfrak{h}$ is a derivation of the Lie algebra if and only if $\left[\bv{2}, \bv{3}\right] = \bv{1}$, $\left[\bv{3}, \bv{1}\right] = \bv{2}$, $\left[\bv{1}, \bv{2}\right] = - \bv{3}$ are eigenvectors for $\tensor{D}\left[\met{g}\right]$ with eigenvalues $d^{1}_{1} = d^{2}_{2} + d^{3}_{3}$, $d^{2}_{2} = d^{3}_{3} + d^{1}_{1}$ and $d^{3}_{3} = d^{1}_{1} + d^{2}_{2}$, respectively.
			The corresponding system of equations admits a solution if and only if $d^{1}_{1} = d^{2}_{2} = d^{3}_{3} = 0$, or equivalently, $\widehat{T}^{1}_{1} = \widehat{T}^{2}_{2} = \widehat{T}^{3}_{3} = \kappa$.
			\end{proof}	
			This allows us to establish the following.
			
			\begin{theorem}
			The Lie group $\mathcal{H} = \widetilde{SL(2)}$ does not support either a left invariant $\mathrm{XCF}$ algebraic soliton or a  left invariant $\mathrm{RG}$-$\mathrm{2}$ algebraic soliton that gives rise to a self-similar solution to the $\mathrm{RG}$-$\mathrm{2}$ flow.
			\end{theorem}
			
			\begin{proof}
				We will first establish that $\widetilde{\mathrm{SL}(2)}$ does not support a left invariant algebraic $\mathrm{XCF}$ soliton.
				By \lemref{SL2lemma}, $\met{g}$ is an algebraic soliton for the $\mathrm{XCF}$ if and only if if $\kappa = \widehat{H}^{1}_{1} = \widehat{H}^{2}_{2} = \widehat{H}^{3}_{3}$.
				According to \eqref{xcfgeneric}, the cross curvature operator is  $ \ccfop = \diag \lp K_{2}K_{3}, K_{1}K_{3}, K_{1}K_{2}\rp$ and it follows that we must have $\kappa = 0$ (and the algebraic soliton is actually a fixed point) or the sectional curvature of the metric $\met{g}$ must be constant.
				 The restrictions on the possible signatures of the Ricci tensor (\cite{Milnor}, Corollary 4.7) show that $\widetilde{\mathrm{SL}(2)}$ does not support a metric of constant sectional curvature.
				 Furthermore, $\kappa = 0$ requires that at least two of three principal sectional curvatures must be zero.
				 From \eqref{ricciorthogonal}, we find that in this case the signature of the Ricci tensor would be $\lp +, +, 0\rp$, $\lp 0, -, -\rp$, or $\lp 0, 0, 0\rp$, none of which are possible for a left invariant metric on $\widetilde{\mathrm{SL}(2)}$.
				 We conclude  that $\widetilde{\mathrm{SL}(2)}$ does not support an algebraic $\mathrm{XCF}$-soliton.\\
				 
				 The proof that $\widetilde{\mathrm{SL}(2)}$ does not support an algebraic soliton that gives rise to self similar solution for the $\mathrm{RG}$-$\mathrm{2}$ flow follows similarly.
				According to \propref{solitonselfsimilar} and \propref{algimpliessol} we only need to concern ourselves with steady algebraic solitons  (i.e., $\kappa = 0$).
				From \lemref{SL2lemma}, we find that $\tensor{D}\left[\met{g}\right] = \rgop\left[\met{g}\right]$ is a derivation of the Lie algebra if $\widehat{RG}^{1}_{1} = \widehat{RG}^{2}_{2} = \widehat{RG}^{3}_{3} = 0$.
				From  \eqref{rg2generic}, it follows that for $\mathbf{D} = \rgop\left[\met{g}\right] $ to be a derivation of the Lie algebra the system of equations
					\begin{align}
					\label{rg2sl2}
					\alpha \left(\left(K_{2}\right)^2 + \left(K_{3}\right)^2\right) &= -2\left(K_{2} + K_{3}\right) = -2 Rc_{11} \nonumber \\
					\alpha \left(\left(K_{3}\right)^2 + \left(K_{1}\right)^2\right) &= -2\left(K_{3} + K_{1}\right) = -2 Rc_{22} \nonumber \\
					\alpha \left(\left(K_{1}\right)^2 + \left(K_{2}\right)^2\right) &= -2\left(K_{1} + K_{2}\right) = -2 Rc_{33}
					\end{align}
					must admit a solution.
					Again, relying on the permissible signatures of the Ricci tensor  of a left invariant Riemannian metric on $\widetilde{\mathrm{SL}(2)}$ ($(+, -, -)$ or $(0, 0, -)$), we conclude that \eqref{rg2sl2} does not admit any solutions.
					 Thus, $\widetilde{SL\left(2\right)}$ does not support an algebraic $\rg$-$2$ soliton structure that gives rise to a self-similar solution of the $\mathrm{RG}$-$\mathrm{2}$ flow.
			\end{proof}
					
		\subsection{$\mathrm{SU}(2)$}
		\label{su2}
			Let $\mathcal{H} = SU(2)$ and let $\met{g}$ be a left invariant metric on $\mathcal{H}$.
			Let $\bv{1}, \bv{2}, \bv{3}$ be a Milnor frame for the left invariant metric $\met{g}$, where
			the Lie algebra structure is then determined by the non-zero brackets 			
			\[
				\left[ \bv{2}, \bv{3}\right] =\bv{1} , \hspace{.5in} 
				\left[\bv{3}, \bv{1} \right] = \bv{2}, \hspace{.5in}
				\left[ \bv{1}, \bv{2}\right] =  \bv{3},
			\]
			and
			$\met{g}  = A \meti{1} + B \meti{2} + C \meti{3}$.
			
			Making the appropriate substitutions into \eqref{sectional} and \eqref{ricciorthogonal}, respectively, we find
			that the principal sectional curvatures of $\met{g}$ are 
			\begin{align}
			\label{sectionalsu2}
			K_{1} &= K\left(\bv{2} \wedge \bv{3}\right) = - \frac{3A^2 - B^2 + 2 B C - C^2 - 2 A B - 2 A C }{4 A  B C} \nonumber\\
			K_{2} & = K\left(\bv{1} \wedge \bv{3}\right) = \frac{ - 3B^2 - 2A C + A^2+ C^2 + 2 A B + 2 B C }{4 A  B C}\nonumber\\			
			K_{3} &= K\left(\bv{1} \wedge \bv{2}\right) = \frac{A^2 + B^2 - 3 C^2 - 2A B + 2B C + 2AC}{4 A B C}.
			\end{align}
			
		In finding the left invariant metrics on $\mathrm{SU}(2)$ that are algebraic solitons giving rise to self similar solutions of the $\mathrm{XCF}$ or  the $\mathrm{RG}$-$\mathrm{2}$ flow, we will again rely on the permissible signatures of the corresponding Ricci tensor, which are stated in \cite{Milnor}.
		Namely, in \cite{Milnor} (Corollary 4.5), Milnor establishes that the signature of the Ricci curvature tensor of a left invariant metric on $\mathrm{SU}(2)$ must be $(+, + , + )$, $(+, 0, 0)$, or $(+, - , -)$.

		\subsubsection{$\mathrm{XCF}$}
			The cross curvature tensor is $\ccf\left[\met{g}\right] = H_{11}\tp{1}{1} + H_{22}\tp{2}{2} + H_{33} \tp{3}{3}$, where
			\begin{equation*}
				H_{11} =  A{K_{2}K_{3}}, \hspace{.5in}
				H_{22} =  B{K_{1}K_{3}}, \hspace{.5in}
				H_{33} = C{K_{1}K_{2}},
			\end{equation*}
		with the corresponding cross curvature operator $\ccfop\left[\met{g}\right] = \left(\widehat{H}^{i}_{j}\right)$  is represented in the given basis by
			\begin{equation}
				\ccfop\left[\met{g}\right] = \diag\left(K_{2}K_{3} , K_{1}K_{3}, K_{1}K_{2}\right).
			\end{equation}
		
		\begin{proposition}
		\label{su2xcf}
		A left invariant metric $\met{g} = A \meti{1} + B\meti{2} + C\meti{3}$ is an algebraic soliton for the $\mathrm{XCF}$ on $\widetilde{SU\left(2\right)}$ if and only if $\met{g}$ has constant sectional curvature (i.e., $A = B = C$).
		\end{proposition}
		Note that \lemref{SL2lemma}  applies equally well to $\mathrm{SU}(2)$ and that the proof carries through without any modifications.
			\begin{proof}
			From \lemref{SL2lemma}, it follows that for  $\mathbf{D} = \ccfop\left[\met{g}\right] - \kappa \mathrm{Id}$ to be a derivation of the Lie algebra we must have $\kappa = \widehat{H}^{1}_{1} = \widehat{H}^{2}_{2} = \widehat{H}^{3}_{3}  $.		
			Since $\widehat{H}^{1}_{1} = K_{2}K_{3}$, $\widehat{H}^{2}_{2} = K_{1}K_{3}$ and $\widehat{H}^{3}_{3} = K_{1}K_{2}$,  $\met{g}$ being an algebraic soliton requires that $\kappa = 0$ or that $\met{g}$ have constant sectional curvature (which only occurs when $A = B = C$).
			That $\met{g}$ is an algebraic soliton when $A = B = C$ is clear.
			To verify that there are no steady algebraic solitons for the $\mathrm{XCF}$, observe that $\kappa = 0$ forces at least two of the three principal sectional curvatures to be equal to zero.
			Since $\ric\left[\met{g}\right] =  A\lp K_{2} + K_{3}\rp \meti{1} B\lp K_{1} + K_{3}\rp \meti{2} +  C\lp K_{1} + K_{2}\rp \meti{3}$,  then having two of the three principal sectional curvatures being zero would force the signature of the Ricci tensor to be $(+, +, 0)$, $(-, -, 0)$, or $(0, 0, 0)$, none of which are possible.			
			Thus, we find that there are no left invariant metrics on $\mathrm{SU}(2)$ where at least two of the three principal sectional curvatures are zero and we conclude that the only algebraic solitons are the metrics of constant sectional curvature.
			\end{proof}
		
			\subsubsection{Algebraic $\rg$-$2$ Solitons}
			
					\begin{proposition}
					A left invariant metric $\met{g} = A \meti{1} + B\meti{2} + C\meti{3}$ is an algebraic soliton that gives rise to a self-similar solution for the $\mathrm{RG}$-$\mathrm{2}$ flow on $SU\left(2\right)$ if and only if
						\begin{enumerate}
							\item $\met{g}$ has constant sectional curvature ($A = B = C$) and $\alpha = -8A$, 
							
							\item $ A = \frac{4}{3}B = \frac{4}{3}C$ and $\alpha = -\frac{9}{2}A$, 
							
							\item $A = B = \frac{3}{4}C$ and $\alpha = -6A$, or
							
							\item $A = C = \frac{3}{4}B$ and $\alpha = -6A$.
						\end{enumerate} 					
					\end{proposition}
					\begin{remark}
					As previously noted, the $\rg$-$2$ flow is not of physical interest when $\alpha $ is negative and all of the algebraic solitons occur when $\alpha $ is negative.
					
					\end{remark}
					
					\begin{proof}
					Again, applying \lemref{SL2lemma}, we find that for a given left-invariant metric $\met{g} = A \meti{1} + B\meti{2} + C \meti{3}$ on $\mathrm{SU}(2)$, $\mathbf{D} = \rgop\left[\met{g}\right]$ is a derivation of the Lie algebra $\mathfrak{h}$ if and only if $\mathbf{D} = 0$.
					Combined with \eqref{rg2generic}, it follows that $\mathbf{D} = \rgop\left[\met{g}\right]$ is a derivation if and only if the system of equations
					\begin{equation}
					\label{rg2su2}
					\begin{cases}
					\medskip
					\alpha \left(\left(K_{2}\right)^2 + \left(K_{3}\right)^2\right) + 2\left(K_{2} + K_{3}\right) = 0\\
					\medskip
					\alpha \left(\left(K_{3}\right)^2 + \left(K_{1}\right)^2\right) + 2\left(K_{3} + K_{1}\right) = 0\\
					\alpha \left(\left(K_{1}\right)^2 + \left(K_{2}\right)^2\right) + 2\left(K_{1} + K_{2}\right) = 0
					\end{cases}
					\end{equation}
					admits a solution.
					The system of equations above is equivalent to 
					\begin{equation}
					\label{equivsys}
					\begin{cases}
					K_{1}\left(\alpha K_{1} + 2\right) &= 0 \\
					K_{2}\left(\alpha K_{2} + 2\right) &= 0 \\
					K_{3}\left(\alpha K_{3} + 2\right) &= 0.
					\end{cases}
					\end{equation}
					As noted in the proof of \propref{su2xcf}, $\mathrm{SU}(2)$ does not support a left invariant metric where two or more of the principal sectional curvatures are equal to zero.
					Thus, the system of equations \eqref{equivsys} is satisfied when the principal sectional curvatures are all equal to $-\frac{2}{\alpha}$, or when one of the principal sectional curvatures $K_{l} = 0$ and the remaining principal sectional curvatures $K_{m}$ and  $K_{n}$ are equal to  $-\frac{2}{\alpha}$.
					In this case, the conditions on the permissible Ricci tensors show that all algebraic solitons will have a Ricci tensor with signature $(+, +, +)$ and $\alpha$ will necessarily be negative.
					The exact solutions can then be obtained by using \eqref{sectionalsu2} to find the appropriate values of $A, B$ and $C$.
					\end{proof}
					
					\begin{remark}
					Note that steady algebraic solitons are actually fixed points for the $\mathrm{RG}$-$\mathrm{2}$ flow on $\mathrm{SU}(2)$.
					\end{remark}

\bibliographystyle{plain}
\bibliography{wears}

\end{document}